\documentclass[12pt]{amsart}
\usepackage{amsfonts}
\usepackage{amsmath}
\usepackage{amssymb}
\usepackage{color}
\usepackage{hyperref}
\usepackage{graphicx}
\usepackage{xspace}
\usepackage{axodraw}
\usepackage{bm}
\usepackage[all]{xy}
 \font \eightrm=cmr8

 \newcommand{\nc}{\newcommand}

\newtheorem{thm}{Theorem}

\newtheorem{cor}[thm]{Corollary}

\newtheorem{lem}[thm]{Lemma}
\newtheorem{prop}[thm]{Proposition}

\newtheorem{rmk}[thm]{Remark}

\def\arbrebbdec#1#2#3{\,{\scalebox{0.60}{ 
  \begin{picture}(48,48) (349,-255)
    \SetWidth{1}
    \SetColor{Black}
    \Vertex(375,-252){12}
    \Line(376,-250)(395,-215)
    \Line(373,-251)(354,-214)
    \Vertex(353,-213){9}
    \Vertex(395,-213){9}
    \SetColor{White}
    \Vertex(375,-252){11}
    \Vertex(353,-213){8}
    \Vertex(395,-213){8}
    \Text(371,-257)[lb]{\large{\Black{$#1$}}}
    \Text(390,-218)[lb]{\large{\Black{$#3$}}}
    \Text(348,-218)[lb]{\large{\Black{$#2$}}}
  \end{picture}
}}\ }

\nc{\ignore}[1]{{}}
\nc{\mrm}[1]{{\rm #1}}
\nc{\dirlim}{\displaystyle{\lim_{\longrightarrow}}\,}
\nc{\invlim}{\displaystyle{\lim_{\longleftarrow}}\,}
\nc{\vep}{\varepsilon} \nc{\ep}{\epsilon}
\nc{\sigmat}{\widetilde\sigma}

\nc{\mchar}{\mrm{Char}}
\nc{\Hom}{\mrm{Hom}}
\nc{\id}{\mrm{id}}

\nc{\remark}{\noindent{\bf{Remark:}}}
\nc{\remarks}{\noindent{\bf{Remarks:}}}

 \nc{\delete}[1]{}
 \nc{\grad}[1]{^{({#1})}}
 \nc{\fil}[1]{_{#1}}

\nc{\BA}{{\Bbb A}} \nc{\CC}{{\Bbb C}} \nc{\DD}{{\Bbb D}}
\nc{\EE}{{\Bbb E}} \nc{\FF}{{\Bbb F}} \nc{\GG}{{\Bbb G}}
\nc{\HH}{{\Bbb H}} \nc{\LL}{{\Bbb L}} \nc{\NN}{{\Bbb N}}
\nc{\PP}{{\Bbb P}} \nc{\QQ}{{\Bbb Q}} \nc{\RR}{{\Bbb R}}
\nc{\TT}{{\Bbb T}} \nc{\VV}{{\Bbb V}} \nc{\ZZ}{{\Bbb Z}}
\nc{\Cal}[1]{{\mathcal {#1}}}
\nc{\mop}[1]{\mathop{\hbox {\rm #1} }}
\nc{\smop}[1]{\mathop{\hbox {\eightrm #1} }}
\nc{\mopl}[1]{\mathop{\hbox {\rm #1} }\limits}
\nc{\frakg}{{\mathfrak g}}
\nc{\g}[1]{{\mathfrak {#1}}}
\def \restr#1{\mathstrut_{\textstyle |}\raise-8pt\hbox{$\scriptstyle #1$}}
\def \srestr#1{\mathstrut_{\scriptstyle |}\hbox to
  -1.5pt{}\raise-4pt\hbox{$\scriptscriptstyle #1$}}
\nc{\wt}{\widetilde}
\nc{\wh}{\widehat}
\nc{\un}{\hbox{\bf 1}}
\nc{\redtext}[1]{\textcolor{red}{\tt #1}}
\nc{\bluetext}[1]{\textcolor{blue}{#1}}
\nc{\comment}[1]{[[{\tt {#1}}]] }
\nc{\R}{{\mathbb R}}

\nc\fleche[1]{\mathop{\hbox to #1 mm{\rightarrowfill}}\limits}
\def\semi{\mathrel{\times}\kern -.85pt\joinrel\mathrel{\raise 1.4pt\hbox{${\scriptscriptstyle |}$}}}

\def\qshu{\!\joinrel{\!\scriptstyle\amalg\hskip -3.1pt\amalg}\,\hskip -8.2pt\hbox{-}\hskip 4pt}

\def\sqshu{\joinrel{\scriptscriptstyle\amalg\hskip -2.5pt\amalg}\,\hskip -7pt\hbox{-}\hskip 4pt}
\def\cto{\joinrel{\hbox{$\nearrow$}\hskip -3.65mm\raise 2pt\hbox{$\nearrow$}}}
\def\scto{\joinrel{\scalebox {0.8}{\hbox{$\nearrow$}}\hskip -2.9mm{\scalebox{0.8}{\raise 2pt\hbox{$\nearrow$}}}}}
\def\sctos{\joinrel{\scalebox {0.5}{\hbox{$\nearrow$}}\hskip -1.9mm{\scalebox{0.5}{\raise 2pt\hbox{$\nearrow$}}}}}
\def\ccto{\joinrel{\!\uparrow\hskip -3.3mm\raise 1mm\hbox{$\uparrow$}\,}}
\def\sccto{\joinrel{\,\scriptstyle\uparrow\hskip -1.5mm\raise 0.7mm\hbox{$\scriptstyle\uparrow$}}}
\def\diagramme #1{\vskip 4mm \centerline {#1} \vskip 4mm}

\begin{document}


\title[Word series substitution and mould composition]
      {A comodule-bialgebra structure for word-series substitution and mould composition}

\author{Kurusch Ebrahimi-Fard}
\address{Department of Mathematical Sciences, Norwegian University of Science and Technology (NTNU), 7491 Trondheim, Norway. {\tiny{(on leave from UHA, Mulhouse, France)}}}
\email{kurusch.ebrahimi-fard@ntnu.no }\urladdr{https://folk.ntnu.no/kurusche/ }

\author{Fr\'ed\'eric Fauvet}
\address{IRMA, Univ.~of Strasbourg et CNRS, 7 rue Descartes, 67084 Strasbourg Cedex, France} \email{frederic.fauvet@math.unistra.fr} 

\author{Dominique Manchon}
\address{Univ.~Blaise Pascal, CNRS-UMR 6620, 3 place Vasar\'ely, CS60026, 63178 Aubi\`ere, France}    
\email{manchon@math.univ-bpclermont.fr}


\begin{abstract} 
An internal coproduct is described, which is compatible with Hoffman's quasi-shuffle product. Hoffman's quasi-shuffle Hopf algebra, with deconcatenation coproduct, is a comodule-Hopf algebra over the bialgebra thus defined. The relation with Ecalle's mould calculus, i.e., mould composition and contracting arborification is precised. 
\end{abstract}


\maketitle


\section{Introduction}
\label{sect:intro}

%
%

\noindent A \textsl{word series} \cite{MS} is a formal linear combination, usually infinite:
\begin{equation}
\label{ws}
	\sum_{\bm\omega\in\Omega^*}M^{\bm\omega}C_{\bm\omega},
\end{equation}
where $\Omega$ is a set, called alphabet, and $\Omega^*$ is the free (associative) monoid of words $\bm\omega$ generated by the letters from $\Omega$. The map $\bm\omega\mapsto C_{\bm\omega}$ is a monoid morphism from $\Omega^*$ into a unital associative algebra $\Cal D$ over some base field $\bm k$. We have to assume that $\Cal D$ is endowed with a topology such that the infinite sum \eqref{ws} is convergent. The coefficients $M^{\bm\omega}$ belong to the base field $\bm k$ or to some unital commutative algebra $\Cal A$ such that $\Cal D$ is an $\Cal A$-algebra. The collection $(M^{\bm\omega})_{\bm\omega\in\Omega^*}$ is called a \textsl{mould} in \cite{E1}, whereas the collection $(C_{\bm\omega})_{\bm\omega\in\Omega^*}$ is called a \textsl{comould}, and the word series \eqref{ws} is the \textsl{mould-comould contraction}\footnote{Note that the comould $C$ is chosen to be a monoid antimorphism in \cite{E1}.}.

\medskip

We will stick to the case where the comould $C$ is tautological, namely $C_{\bm\omega}=\bm\omega$. This makes sense with $\Cal D=\bm k\langle\!\langle\Omega\rangle\!\rangle$ being the algebra of noncommutative power series with variables in $\Omega$. The \textsl{mould calculus} has been developed by J.~Ecalle in \cite{E1}, as a powerful tool in studying formal or analytic local objects (vector fields or diffeomorphisms) around the origin in $\mathbb R^n$. In most of the situations encountered, the alphabet $\Omega$ is a commutative semigroup, typically the positive integers, $\Omega=\mathbb N_{>0}=\{1,2,3,\ldots\}$ or $\Omega=\mathbb N_{>0}^n$. In this case, two associative products are available on the vector space of moulds: the \textsl{mould product} $\times$ and the \textsl{mould composition} $\circ$. The definition of the latter involves the semigroup structure of $\Omega$ in an essential way (see the definitions in Section \ref{sect:moules}).\\

The algebra $\bm k\langle\!\langle\Omega\rangle\!\rangle$ is the dual of the coalgebra $\bm k\langle\Omega\rangle$ of noncommutative polynomials with variables in $\Omega$, endowed with the deconcatenation coproduct:
$$
	\Delta(\bm\omega):=\sum_{\bm\omega'.\bm\omega''=\bm\omega}\bm\omega'\otimes\bm\omega''.
$$
Moreover, a mould gives rise, by linear extension, to a unique linear form on $\bm k\langle\Omega\rangle$. The identification of the vector space of moulds with $\bm k\langle\!\langle\Omega\rangle\!\rangle$ is achieved through the map:
$$
	M \mapsto W^M:=\sum_{\bm\omega\in\Omega^*}M ^{\bm\omega}\bm\omega,
$$
which associates to each mould $M$ its corresponding word series $W^M \in \bm k\langle\!\langle\Omega\rangle\!\rangle$. It is well-known \cite{H2} that $\Cal H^{\Omega} = (\bm k\langle\Omega\rangle,\,\qshu,\Delta)$ is a commutative Hopf algebra, where $\,\qshu$ is Hoffman's quasi-shuffle product, recursively defined by $\bm\omega\qshu\un = \un\qshu\bm\omega = \bm\omega$ (here $\un$ stands for the empty word) and:
$$
	a\bm\omega'\qshu b\bm\omega''
	:= a(\bm\omega'\qshu b\bm\omega'')
		+b(a\bm\omega'\qshu \bm\omega'')
			+[a+b](\bm\omega'\qshu \bm\omega'').
$$
Here $a$ and $b$ are letters in $\Omega$ and $\bm\omega',\bm\omega''$ are words in $\Omega^*$. The notation $[a+b]$ stands for the internal sum of the two letters in the commutative semigroup $\Omega$. The Hopf algebra $\Cal H^{\Omega}$ is $(\Omega\sqcup\{0\})$-graded by the weight defined by $||\un||:=0$ and:
$$
	 ||\bm\omega|| := [\omega_1+\cdots+\omega_\ell] \in \Omega
$$
for a word $\bm\omega = \omega_1\cdots\omega_\ell \in \Omega^*$ of length $|\bm\omega|:=\ell$. The mould product $\times$ of \cite{E1} is hence obtained by dualizing the deconcatenation coproduct $\Delta$, and thus reflects the noncommutative concatenation product in $\bm k\langle\!\langle\Omega\rangle\!\rangle$, namely:
$$
	W^M.W^N = W^{M\times N}.
$$

Inspired by mould composition, i.e., substitution of noncommutative formal series, we exhibit in this paper a second coproduct $\Gamma$ on $\Cal H^{\Omega}$ which is coassociative and compatible with the quasi-shuffle product. It is \textsl{internal} in the sense that it respects each homogeneous component with respect to the weight grading. Moreover, it endows Hoffman's quasi-shuffle Hopf algebra $(\Cal H^{\Omega},\,\qshu,\Delta)$ with the structure of comodule-Hopf algebra \cite{M77} over the bialgebra $(\Cal H^{\Omega},\,\qshu,\Gamma)$ (Theorem \ref{coprod-c}). Surprisingly enough, dualizing the internal coproduct $\Gamma$ gives rise to a second composition product $\diamond$ on moulds which does not coincide with the mould composition $\circ$ in general. More precisely the identity $M\diamond N=M\circ N$ holds when the mould $N$ is \textsl{symmetrel} \cite{E92}, i.e., when the identity
$$
	N^{\bm\omega'\,\sqshu\bm\omega''} 
		= N^{\bm\omega'} N^{\bm\omega''}
$$
holds for any words $\bm\omega',\bm\omega''\in\Omega^*$. The composition $\circ$ distributes over the mould product $\times$ on the right, whereas the interplay between the product $\diamond$ and the mould product $\times$ is described by the comodule-Hopf algebra structure.

\smallskip

The paper is organized as follows. In Section \ref{sect:moules} we recall some basics of J.~Ecalle's mould calculus.  Our main result (Theorem \ref{coprod-c}) is proved in Section \ref{sect:qs} via generalized quasi-symmetric functions, by means of sum and product of two auxiliary totally ordered alphabets \cite{KLT97, NPT13}. Section \ref{sect:wqsh} is devoted to the notion of weak quasi-shuffles, which are surjective maps generalizing quasi-shuffles. They are then used to provide an alternative, more pedestrian proof of Theorem \ref{coprod-c}. A link with contracting arborification \cite{EV} is investigated in Section \ref{arborification}. In particular, we prove (in Theorem \ref{arbo-interne}) that contracting arborification changes the composition $\diamond$ into an analogous composition $\diamond$ of arborescent moulds, obtained by dualizing an internal coproduct which is a straightforward decorated version of the one given in \cite{CEM}. The analog of the mould composition $\circ$ for arborescent moulds was given in \cite{EV95}. Its relation with composition $\circ$ of ordinary moulds via contracting arborification was precised in \cite{Me06}.\\

\smallskip

\noindent\textbf{Acknowledgements}: We thank Fr\'ed\'eric Patras for interesting discussions at an early stage of the paper, and Jean-Yves Thibon for crucial illuminating remarks about manipulations of alphabets. We thank the two anonymous referees for their pertinent remarks which lead to substantial improvement of the paper. Work partially supported by Agence Nationale de la Recherche, projet CARMA  NR-12-BS01-0017.

\section{Background on mould calculus}
\label{sect:moules}

Mould calculus evolved as part of Ecalle's resurgence theory, and consists of a combinatorial setting which provides explicit as well as efficient formulas for studying local properties of dynamical systems. In this section we recall some basic facts on mould calculus. See references \cite{C05,E1} for more details.

\subsection{The algebraic setting}
\label{ssect:setting}

Let $\Omega$ be an alphabet endowed with a commutative semigroup law written additively. For example we can choose positive integers, i.e., $\Omega=\NN_{>0}=\{1,2,3,\ldots\}$. A word $\bm\omega$ consists of a string of letters $\omega_i \in \Omega$, and will be denoted:
$$
	\bm\omega=\omega_1\cdots \omega_n.
$$
The \textsl{length} of $\bm \omega$ is given by its number $|\bm\omega|=n$ of letters. The \textsl{weight} of the word $\bm\omega$ is defined to be the sum of its letters in $\Omega$:
\begin{equation}
\label{mould-weight}
	\|\bm\omega\|:=\left[\sum_{i=1}^n\omega_i\right] \in \Omega,
\end{equation}
where the brackets indicate the internal sum in the commutative semigroup $\Omega$, in contrast with formal linear combinations which will be widely used in the sequel. Hence the weight takes its values in $\Omega\sqcup\{0\}$. The unique word of weight zero is the empty word, denoted by $\un$, which happens to be of length zero, i.e., $|\un|=\|\un\|=0$. The concatenation of two words $\bm\omega=\omega_1\cdots \omega_p$ and $\bm\omega'=\omega_{p+1}\cdots \omega_{p+q}$ is defined to be the word:
$$
	\bm\omega.\bm\omega'=\omega_{1}\cdots \omega_{p+q}
$$
of length $|\bm\omega.\bm\omega'|=p+q$. It defines a noncommutative, associative and unital product, with the unit being the empty word. We denote by $\Omega^*$ the monoid of words on $\Omega$ thus defined. Let $\Cal H^{\Omega}$ be the vector space (over some base field $\bm k$) spanned by the elements of $\Omega^*$. A \textsl{mould} on the alphabet $\Omega$ is a linear form $M$ on $\Cal H^{\Omega}$ (or, more generally, a linear map from $\Cal H^{\Omega}$ into some unital commutative $\bm k$-algebra $\Cal A$). Note that in the literature a mould is sometimes denoted ${M}^{\bullet}$. The evaluation of $M$ at a word $\bm\omega$ will be denoted by $M^{\bm\omega} \in \Cal A$. For two moulds $N,M$ we recall the definitions of {\it{mould multiplication}} and {\it{mould composition}}, respectively:
\allowdisplaybreaks
\begin{eqnarray}
	(M\times N)^{\bm\omega}&=&\sum_{\bm\omega'.\bm\omega''
	=\bm\omega}M^{\bm\omega'}N^{\bm\omega''},\label{mp}\\
	(M\circ N)^{\bm\omega}&=&\sum_{s\ge 1}\sum_{\bm{\omega}
	=\bm{\omega}^1.\, \cdots .\, \bm{\omega}^s} 
	M^{\|\bm\omega^1\|\cdots\|\bm\omega^s\|}N^{\bm\omega^1}\cdots N^{\bm\omega^s}.\label{mc}
\end{eqnarray}
Recall that for $1 \le i \le s$ the weight $\|\bm\omega^i\|$ is a letter in $\Omega$. The basic algebraic properties of mould calculus can be stated as follows \cite{E1}.

\begin{prop}\label{mould-calculus}
Mould multiplication and composition are both associative and noncommutative. Composition distributes on the right over multiplication, namely:
$$
	(M\times M')\circ N=(M\circ N)\times(M'\circ N)
$$
for any triple of moulds $(M,M',N)$. The unit for mould multiplication is the mould $\varepsilon$ defined by $\varepsilon^{\bm 1}=1$ and $\varepsilon^{\bm\omega}=0$ for any nontrivial word $\bm\omega \in \Omega^*$. The unit for mould composition is the mould $I$ defined by $I^\omega=1$ for any letter $\omega \in \Omega$ and $I^{\bm\omega}=0$ for $\bm\omega=\un$ or length $|\bm\omega |\ge 2$.
\end{prop}

\begin{proof}
The unit properties for $\varepsilon$ and $I$ as well as the noncommutativity of both products are immediate. Whereas the associativity of the mould multiplication is easily checked (it is nothing but the convolution product dual to the deconcatenation coproduct), the two other properties involving mould composition are better seen when the latter is interpreted as a substitution of alphabets. Indeed, any $\Cal A$-valued mould $M$ gives rise to a \textsl{word series} \cite{MS}:
\begin{equation}\label{word-series}
	W^M:=\sum_{\bm\omega\in\Omega^*}M^{\bm \omega}\bm \omega,
\end{equation}
which obviously determines the mould $M$ in return. The space of noncommutative formal series with variables in $\Omega$ (word series) and with coefficients in $\Cal A$ is denoted by $\Cal A\langle\!\langle\Omega\rangle\!\rangle$. The subspace of (noncommutative) polynomials is denoted $\Cal A\langle\Omega\rangle$. The homogeneous components of $W^M$ with respect to weight defined in \eqref{mould-weight} are given for any letter $\kappa\in\Omega$ by:
\begin{equation}\label{hc}
	\iota^M(\kappa):=\sum_{\bm\omega\in\Omega^* \atop ||\bm\omega||
				 =\kappa}M^{\bm \omega}\bm \omega.
\end{equation}
This gives rise to a linear map $\iota^M:\Omega\to\Cal A\langle\!\langle\Omega\rangle\!\rangle$, which uniquely extends by $\Cal A$-linearity, multiplicativity and completion, to a unital $\Cal A$-algebra endomorphism $\jmath^M:\Cal A\langle\!\langle\Omega\rangle\!\rangle\to\Cal A\langle\!\langle\Omega\rangle\!\rangle$. Remark that the word series of the mould $I$ is given by the formal sum of the letters in $\Omega$:
\begin{equation}\label{comp-id}
	W^I=\sum_{\omega\in\Omega}\omega,
\end{equation}
such that $\jmath^I=\mop{Id}$. From \eqref{word-series}, \eqref{hc} and \eqref{comp-id} we immediately get for any mould $M$ its corresponding word series:
\begin{equation}\label{word-series-bis}
	W^M=\jmath^M(W^I).
\end{equation}

\begin{lem}\label{Wjmath}
Let $M,N$ be two $\Cal A$-valued moulds on the alphabet $\Omega$, where $\Cal A$ is a commutative unital $\bm k$-algebra. Then:
\begin{enumerate}
	\item $W^{M\times N}=W^M.W^N$,
	\item $\jmath^N\circ\jmath^M=\jmath^{M\circ N}$.
\end{enumerate}
\end{lem}

\begin{proof}
Proving the first assertion is straightforward:
\allowdisplaybreaks{
\begin{eqnarray*}
	W^{M\times N}
	&=&\sum_{\bm \omega\in\Omega^*}(M\times N)^{\bm\omega}\bm\omega\\
	&=&\sum_{\bm \omega\in\Omega^*}\sum_{\bm\omega'.\bm \omega''
		=\bm\omega}M^{\bm\omega'}N^{\bm\omega''}\bm\omega\\
	&=&\sum_{\bm\omega',\bm\omega''\in\Omega^*}M^{\bm\omega'}N^{\bm\omega''}\bm\omega'.\bm\omega''\\
	&=&W^M.W^N.
\end{eqnarray*}}
Now let $\kappa$ be any letter in $\Omega$, and compute:
\allowdisplaybreaks{
\begin{eqnarray*}
	\jmath^N\circ\jmath^M(\kappa)
	&=&\jmath^N\Big(\sum_{\bm\omega\in\Omega^* \atop ||\bm\omega||=\kappa}M^{\bm\omega}\bm\omega\Big)\\
	&=&\sum_{\bm\omega\in\Omega^* \atop ||\bm\omega||=\kappa}M^{\bm\omega}\jmath^N(\bm\omega)\\
	&=&\sum_{r\ge 1}\sum_{\bm\omega\in\Omega^* \atop {||\bm\omega||=\kappa,\ |\bm\omega|=r}}
		M^{\bm\omega}\jmath^N(\omega_1)\cdots\jmath^N(\omega_r)\\
	&=&\sum_{r\ge 1}\sum_{\bm\omega\in\Omega^* \atop \big[||\bm\omega^{1}||+\cdots+||\bm\omega^{r}||\big]=\kappa}
		M^{||\bm\omega^{1}||\cdots||\bm\omega^{r}||}N^{\bm\omega^{1}}\cdots 
				N^{\bm\omega^{r}}\bm\omega^{1}\cdots\bm\omega^{r}\\
	&=&\sum_{r\ge 1}\sum_{\bm\omega\in\Omega^* \atop ||\bm\omega||=\kappa}
	\Big(\sum_{\bm\omega^{1}\cdots\bm\omega^{r}=\bm\omega} M^{||\bm\omega^{1}||\cdots||\bm\omega^{r}||}
			N^{\bm\omega^{1}}\cdots N^{\bm\omega^{r}}\Big)\bm\omega\\
	&=&\jmath^{M\circ N}(\kappa).
\end{eqnarray*}}%
Both $\jmath^{M\circ N}$ and $\jmath^{N}\circ\jmath^M$ are algebra morphisms that coincide on letters from $\Omega$, hence they are equal.
\end{proof}

\noindent\textit{Proof of Proposition \ref{mould-calculus} {\rm{(continued)}}}: from \eqref{word-series}, \eqref{hc}, \eqref{comp-id} and Lemma \ref{Wjmath} we have for any moulds $M,N,P$:
\begin{equation}
	W^{M\circ (N\circ P)}=\jmath^P\circ\jmath^N\circ\jmath^M(W^I)=W^{(M\circ N)\circ P},
\end{equation}
and for any moulds $M,M',N$:
\allowdisplaybreaks{
\begin{eqnarray*}
	W^{(M\times M')\circ N}&=&\jmath^N\circ\jmath^{M\times M'}(W^I)=\jmath^N(W^{M\times M'})\\
	&=&\jmath^N(W^M).\jmath^N(W^{M'})=\jmath^N\circ\jmath^M(W^I).\jmath^N\circ\jmath^{M'}(W^I)\\
	&=&W^{M\circ N}.W^{M'\circ N}\\
	&=&W^{(M\circ N)\times(M'\circ N)}.
\end{eqnarray*}}
\end{proof}

Let us recall for later use that a mould $M$ is called \textsl{symmetrel} if it respects the quasi-shuffle product, i.e., if for any  words $\bm\omega, \bm\omega' \in \Omega^*$:
\begin{equation}
\label{symmetrel}
	M^{\bm\omega\,\sqshu\!\bm\omega'}=M^{\bm\omega}M^{\bm \omega'},
\end{equation}
where the quasi-shuffle product $\,\qshu$ of words is recalled in Section \ref{sect:main} (with reference to the next section) below. A mould is \textsl{symmetral} if it respects the ordinary shuffle of words. The \textsl{Hoffman exponential} \cite{H2} establishes a bijection from shuffle onto quasi-shuffle Hopf algebra, hence from symmetral onto symmetrel moulds. The latter can be expressed as mould composition with the \textsl{exponential mould} defined by:
\begin{equation}
\exp^{\bm\omega}:=\frac{1}{\vert\bm\omega\vert!},
\end{equation}
i.e., $M$ is symmetral if and only if $M\circ\mop{exp}$ is symmetrel \cite[Paragraph 2.1.12]{E02}.

\subsection{On geometric growth condition}
\label{ssect:growth}

In the applications of mould calculus to dynamical systems, for which J.~Ecalle had invented and developed this powerful formalism, the mould operations, and specifically the{\sl{ interplay}} between mould composition and product, are crucial in many occurrences for obtaining important results. Notably the ones pertaining to the growth properties of the moulds involved. For the analyst, indeed, what is at stake is eventually the convergence of expansions containing, say, a complex valued mould $M=M^{\bullet}$ , indexed, e.g., by $\Omega= \mathbb{N}_{>0}$ or $\Omega=\mathbb R_{>0}$, which thus must satisfy estimates of the following type:
\begin{equation}
\label{growth}
	| M^{\bm \omega} | \leqslant C \kappa^{\|\bm\omega\|} \hspace{2em}
\end{equation}
with $C, \kappa \in \mathbb {R}_{>0}$. If $N$ is another mould verifying geometric growth condition \eqref{growth} with constants $C'$ and $\kappa'$, elementary computations show that $M \times N$ and $M \circ N$ also grow geometrically in the case $\Omega= \mathbb N_{>0}$:
\begin{eqnarray}
	\vert(M\times N)^{\bm\omega}\vert&\leqslant& 
	CC'(\vert\bm\omega\vert+1) \big(\mop{max}(\kappa,\kappa')\big)^{\|\bm\omega\|} 
	\leqslant CC'(\|\bm\omega\|+1) \big(\mop{max}(\kappa,\kappa')\big)^{\|\bm\omega\|},\\
	\vert(M\circ N)^{\bm\omega}\vert&\leqslant& 
	C(1+C')^{\vert\bm \omega\vert-1}(\kappa\kappa')^{\|\bm\omega\|}\leqslant C\big((1+C')\kappa\kappa'\big)^{\|\bm\omega\|}.
\end{eqnarray}

Compositional inversion and compositional logarithm, however, do not preserve geometrical growth. Hence, to prove such a property for a given mould, some intermediate key moulds are quite often obtained, with a closed form expression that makes it possible to verify straightforwardly the geometrical growth. A clever use of product and composition can then rather easily yield the sought after property for the other moulds, which are connected to the ones for which geometrical growth is already established.

Several sophisticated examples of such a scheme can be found, e.g., in reference \cite{E02}. As an illustration, we extract the following simple example from section 6 of the aforementioned article. For matters of resummation of real analytic divergent series, two symmetrel moulds ${rem}^{\bullet}$ and ${lem}^{\bullet}$ (indexed by sequences of positive numbers, i.e., $\Omega=\mathbb R_{>0}$, and with values in $\mathbb {C}$) are considered, which are bound by the following relation. Note that we follow the notation used in \cite[Paragraph 7.5]{E02}, in which $J^{\bullet}$ designates the elementary symmetrel mould $J^{ \omega_1 \cdots \omega_r} = ( - 1)^r$, and:
\[ 
	{lem}^{\bullet} = ({rem}^{\bullet} \circ J^{\bullet}) \times J^{\bullet} 
\]
With such a formula, the growth properties of ${rem}^{\bullet}$ and ${lem}^{\bullet}$ are clearly connected and, at the level of arborescent moulds (see Section \ref{arborification} below), such relations remain valid and enable to avoid calculations which
would otherwise be quite intractable without the combined use of the $\times$ and $\circ$ operations.

We remark that, in the language of Hopf algebras, characters with good growth properties have very recently been systematically studied in particular in \cite{BS}, where the notion of tame characters has been introduced.

\section{Quasi-symmetric functions and totally ordered alphabets}
\label{sect:qs}

\subsection{Hoffman's quasi-shuffle Hopf algebra}
\label{ssect:HoffmanHopf}

Let $p,q,r$ be three nonnegative integers, with $p,q \ge 1$ and $r \le p+q-1$. We denote by $\mop{qsh}(p,q;r)$ the set of \textsl{$(p,q)$-quasi-shuffles of type $r$}, i.e., surjective maps:
$$
	\sigma:\{1,\ldots,p+q\}\fleche 8\hskip -5.8mm\fleche 6\{1,\ldots,p+q-r\}
$$
subject to the conditions $\sigma_1<\cdots <\sigma_p$ and $\sigma_{p+1}<\cdots <\sigma_{p+q}$. Quasi-shuffles of type $r=0$ are the ordinary $(p,q)$-shuffles, which are the permutations of the set $\{1,\ldots,p+q\}$ which display the $p$ first (resp.~$q$ last) elements in increasing order.\\

We keep the notations of Section \ref{sect:moules}. The \textsl{quasi-shuffle product} of two words $\bm\omega=\omega_1\cdots\omega_p$ and $\bm\omega'=\omega_{p+1}\cdots\omega_{p+q}$ in $\Cal H^{\Omega}$ is defined by the formal sum:
\begin{equation}\label{quasi-shuffle}
	\omega_1\cdots\omega_p\qshu\omega_{p+1}\cdots\omega_{p+q}
	= \sum_{r\ge 0}\sum_{\sigma\in\smop{qsh}(p,q;r)}\omega^\sigma_1\cdots\omega^\sigma_{p+q-r},
\end{equation}
with $\omega^\sigma_k:=\left[\sum_{\sigma_j=k}\omega_j\right] \in \Omega$. Note that the sum inside the brackets contains either one or two terms. The product \eqref{quasi-shuffle} is associative as well as commutative, and has  the empty word $\un$ as unit. A more immediate way of calculating quasi-shuffle products of words is given in terms of the equivalent recursive definition:
\allowdisplaybreaks
\begin{align*}
	\bm\omega\qshu\bm\omega' &=
	\omega_1(\omega_2\cdots\omega_p\qshu\omega_{p+1}\cdots\omega_{p+q}) 
	+ \omega_{p+1}(\omega_1\cdots\omega_p\qshu\omega_{p+2}\cdots\omega_{p+q})\\
	&\qquad\ + [\omega_1 + \omega_{p+1}](\omega_2\cdots\omega_p\qshu\omega_{p+2}\cdots\omega_{p+q}).
\end{align*} 
For example, 
\allowdisplaybreaks
\begin{align*}
	\omega_1\qshu\omega_{2}=\omega_2\qshu\omega_{1}
	&=\omega_1\omega_{2} + \omega_2\omega_{1} + [\omega_1+\omega_{2}],\\
	\omega_1\omega_{2}\qshu\omega_{3}=\omega_3\qshu\omega_1\omega_2
	&=\omega_1(\omega_{2}\qshu\omega_{3}) + \omega_{3}\omega_1\omega_{2} + [\omega_1+\omega_{3}]\omega_{2}\\
	&=\omega_1\omega_2\omega_3+\omega_1\omega_3\omega_2+\omega_{3}\omega_1\omega_{2} +\omega_1[\omega_2+\omega_3]+ [\omega_1+\omega_{3}]\omega_{2}.
\end{align*}
\ignore{In the case of an alphabet $\Omega$ endowed with a trivial semigroup law, $[\omega_i+\omega_j]=0$, the quasi-shuffle product \eqref{quasi-shuffle} reduces to the usual shuffle product \cite{reutenauer}. The latter is defined in terms of particular bijections, i.e., the $(p,q)$-shuffles mentioned in the foregoing section.\\}

\noindent Let $\Delta$ be the \textsl{deconcatenation coproduct} defined on words  $\bm\omega=\omega_1\cdots\omega_p \in \Omega^*$:
\begin{equation}
\label{deconcat}
	\Delta(\omega_1\cdots\omega_p)
	=	\omega_1\cdots\omega_p \otimes \un + \un \otimes \omega_1\cdots\omega_p 
		+\sum_{j=1}^{p-1}\omega_1\cdots\omega_j\otimes\omega_{j+1}\cdots\omega_p.
\end{equation}
It is well-known that $(\Cal H^{\Omega},\,\qshu,\Delta)$ is a commutative, noncocommutative, connected Hopf algebra, graded by the weight defined in \eqref{mould-weight}. See \cite{H2} for details.

\subsection{The internal coproduct and statement of the main result}
\label{ssect:newCoprod}

Inspired by mould composition, we introduce the \textsl{decomposition coproduct} $\Gamma$ on $\Cal H^{\Omega}$, given for any word $\bm\omega \in \Omega^*$ by:
\begin{equation}\label{dual-comp}
	\Gamma(\bm\omega)
	:=\sum_{s\ge 1}\sum_{\bm{\omega}=\bm{\omega}^1.\, \cdots .\, \bm{\omega}^s}
	\|\bm\omega^1\|\cdots\|\bm\omega^s\|\otimes \bm\omega^1\qshu\cdots\qshu\bm\omega^s.
\end{equation}
\begin{thm}\label{coprod-c}
The coproduct $\Gamma: \Cal H^{\Omega} \to \Cal H^{\Omega} \otimes \Cal H^{\Omega}$ is coassociative, noncocommutative and compatible with the quasi-shuffle product {\rm $\,\qshu$}. Moreover, Hoffman's quasi-shuffle Hopf algebra {\rm $(\Cal H^{\Omega},\,\qshu,\Delta)$} is a right comodule-Hopf algebra on the bialgebra {\rm $(\Cal H^{\Omega},\,\qshu,\Gamma)$}, in the sense that the following diagrams commute:
{\rm
\diagramme{
\xymatrix{\Cal H^{\Omega}\ar[rr]^{\Gamma}\ar[d]^{\Delta}&&\Cal H^{\Omega}\otimes\Cal H^{\Omega}\ar[d]^{\Delta\otimes\smop{Id}}\\
\Cal H^{\Omega}\otimes\Cal H^{\Omega}\ar[d]^{\Gamma\otimes\Gamma}&&\Cal H^{\Omega}\otimes\Cal H^{\Omega}\otimes\Cal H^{\Omega}\\
\Cal H^{\Omega}\otimes\Cal H^{\Omega}\otimes\Cal H^{\Omega}\otimes\Cal H^{\Omega}\ar[rr]^{\tau_{23}}&&\Cal H^{\Omega}\otimes\Cal H^{\Omega}\otimes\Cal H^{\Omega}\otimes\Cal H^{\Omega}\ar[u]_{\smop{Id}\otimes\smop{Id}\otimes\;\sqshu}
}}
\diagramme{
\xymatrix{\Cal H^\Omega\ar[r]^\Gamma\ar[d]_{\varepsilon} & \Cal H^\Omega\otimes\Cal H^\Omega\ar[d]^{\varepsilon\otimes \smop{Id}}\\
\bm k\ar[r]_u & \Cal H^{\Omega}
}
\hskip 20mm
\xymatrix{\Cal H^\Omega\ar[r]^\Gamma\ar[d]_{S} & \Cal H^\Omega\otimes\Cal H^\Omega\ar[d]^{S\otimes \smop{Id}}\\
\Cal H^{\Omega}\ar[r]_\Gamma & \Cal H^\Omega\otimes\Cal H^\Omega
}
}
}
\noindent where all arrows are algebra morphisms for the quasi-shuffle product \rm$\,\qshu$.
\end{thm}

\subsection{Proof of Theorem \ref{coprod-c} via generalized quasi-symmetric functions and totally ordered alphabets}\label{sect:qstoa}

In the case $\Omega=\mathbb N_{>0}=\{1,2,3,\ldots\}$, the quasi-shuffle Hopf algebra $(\Cal H^{\Omega},\,\qshu,\Delta)$ admits a polynomial realization making it isomorphic to  the Hopf algebra $\mop{\bf QSym}$ of quasi-symmetric functions (on some infinite alphabet $X$). Following an idea by Novelli, Patras and Thibon \cite[Section 8]{NPT13}, this generalizes to any commutative semigroup $\Omega$ provided the Hopf algebra $\mop{\bf QSym}$ is replaced by the Hopf algebra $\mop{\bf QSym}^\Omega$ of $\Omega$-quasi-symmetric functions, the definition of which is recalled below. Theorem  \ref{coprod-c} can then be derived from simple manipulations on alphabets. The internal coproduct on $\mop{\bf QSym}^\Omega$ is related to the Tits product on set compositions of a given finite set (\cite{BZ09,B00}, see also \cite{AFM15, AM10, FFM15}).\\

Both coproducts on $\mop{\bf QSym}^\Omega$ can be described by means of the alphabet technique developed in \cite{KLT97}. Indeed, let $X$ be an auxiliary alphabet, supposed to be infinite and totally ordered. Let $\bm k^\Omega[[X]]$ be the $\bm k$-vector space of formal series with indeterminates in $X$ and exponents in $\Omega$, i.e., formal sums:
\begin{equation}
\label{series}
	\sum_{P \subset X \atop |P| < \infty}\,\sum_{\nu:P\to\Omega}\lambda_{P,\nu}\prod_{x\in P}x^{\nu(x)},
\end{equation}
where the coefficients $\lambda_{P,\nu}$ belong to the base field $\bm k$. The commutative multiplication on $\bm k^\Omega[[X]]$ is determined by the rule:
\begin{equation}
	x^\omega x^{\omega'}=x^{[\omega+\omega']}
\end{equation}
for any $x\in X$ and $\omega,\omega'\in\Omega$. A series in $\bm k^\Omega[[X]]$ is quasi-symmetric if, for any word $\bm\omega=\omega_1\cdots\omega_r\in \Omega^*$, the coefficient in front of $x_1^{\omega_1}\cdots x_r^{\omega_r}$ is the same for any $x_1<\cdots <x_r\in X$. The vector space of $\Omega$-quasi-symmetric functions is denoted by $\mop{\bf QSym}^\Omega(X)$. The $\Omega$-quasi-symmetric functions $Q_{\bm \omega}\in\bm k^\Omega[[X]]$ defined by:
\begin{equation}
	Q_{\bm \omega}(X):=\sum_{x_1<\cdots <x_r\in X}x_1^{\omega_1}\cdots x_r^{\omega_r}
\end{equation}
form a linear basis\footnote{Here we use the symbol $Q$ for quasi-symmetric functions instead of the usual notation $M$ which would be in conflict with the notations for moulds.} of $\mop{\bf QSym}^\Omega(X)$. One can easily prove the following identity for any words $\bm \omega',\bm\omega''\in \Omega^*$:
\begin{equation}
	Q_{\bm\omega'}(X)Q_{\bm\omega''}(X)=Q_{\bm\omega'\sqshu\bm\omega''}(X),
\end{equation}
making $\mop{\bf QSym}^\Omega(X)$ a commutative algebra  isomorphic to the quasi-shuffle algebra $(\bm k\langle\Omega\rangle,\, \qshu)$. Now let $Y$ be another infinite and totally ordered alphabet. We denote by $X+Y$ the ordinal sum of $X$ and $Y$, defined as the disjoint union $X\sqcup Y$ endowed with the unique total order which restricts to the total orders of $X$ and $Y$, and such that any element of $Y$ is bigger than any element of $X$. Let us also consider the product $XY$, defined as the cartesian product $X\times Y$ endowed with the lexicographical order\footnote{The sum and product of alphabets thus defined are associative, but obviously not commutative.}. One can compute:
\begin{eqnarray*}
	Q_{\bm\omega}(X+Y)&=&\sum_{z_1<\cdots<z_r\in X+Y}z_1^{\omega_1}\cdots z_r^{\omega_r}\\
	&=&\sum_{s=0}^r\sum_{x_1<\cdots<x_s\in X}\sum_{y_1<\cdots <y_{r-s}\in Y}x_1^{\omega_1}\cdots 
	x_s^{\omega_s}y_1^{\omega_{s+1}}\cdots y_{r-s}^{\omega_r}\\
	&=&\sum_{\bm\omega=\bm\omega'\bm\omega''}Q_{\bm\omega'}(X)Q_{\bm\omega''}(Y),
\end{eqnarray*}
as well as:
\begin{eqnarray*}
	Q_{\bm\omega}(XY)&=&\sum_{z_1<\cdots<z_r\in XY}z_1^{\omega_1}\cdots z_r^{\omega_r}\\
	&=&\sum_{(x_1,y_1)<\cdots <(x_r,y_r)\in XY}x_1^{\omega_1}y_1^{\omega_1}\cdots x_r^{\omega_r}y_r^{\omega_r}\\
	&=&\sum_{s=1}^r\sum_{\bm \omega^1\cdots\bm\omega^s=\bm\omega}\sum_{x_1<\cdots<x_s\in X}
	x_1^{\|\bm\omega^1\|}\cdots x_s^{\|\bm\omega^s\|} Q_{\bm\omega^1}(Y)\cdots Q_{\bm\omega^s}(Y)\\
	&=&\sum_{s=1}^r\sum_{\bm \omega^1\cdots\bm\omega^s=\bm\omega}Q_{\|\bm\omega_1\|\cdots\|\bm\omega_s\|}(X)Q_{\bm\omega^1}(Y)\cdots Q_{\bm\omega^s}(Y).
\end{eqnarray*}
In the second computation, any letter $(x,y)\in XY$ is identified with the product $xy$. Now if $Y$ is chosen to be a copy of $X$, both expressions can be seen as elements in $\mop{\bf QSym}^\Omega(X)\otimes\mop{\bf QSym}^\Omega(X)$, thus defining two coproducts $\overline\Delta$ and $\overline\Gamma$ on $\mop{\bf QSym}^\Omega(X)$. We obviously have:
$$
	\overline\Delta(Q_{\bm\omega})=(Q\otimes Q)_{\Delta\bm\omega},
	\hskip 12mm 
	\overline\Gamma(Q_{\bm\omega})=(Q\otimes Q)_{\Gamma\bm\omega}
$$
where $\Delta$ is the deconcatenation and $\Gamma$ is the internal coproduct defined earlier. Hence coassociativity of $\Delta$ and $\Gamma$ can be directly derived from the associativity of the sum (resp.~product) of alphabets. Compatibility with the product (which amounts to quasi-shuffle product on words) is naturally given. The comodule-Hopf algebra structure is derived from the natural isomorphism of totally ordered sets between $(X+Y)Z$ and $XZ+YZ$, and from the fact that the antipode is given by replacing alphabet $X$ with $-X$.

\subsection{Remarks on the internal coproduct}
\label{ssect:remarks}

\begin{rmk}\rm
Coassociativity of the coproduct $\Gamma$ can also directly be derived by duality from the associativity of mould composition. Indeed, we have $(M\otimes N)^{\Gamma\bm\omega}=(M\circ N)^{\bm \omega}$ when the moulds $M$ and $N$ are symmetrel. Hence for any word $\bm\omega\in\Omega^*$ we have:
\begin{equation}\label{comp-assoc}
	(M\otimes N\otimes P)^{[(\Gamma\otimes\smop{Id})\Gamma-(\smop{Id}\otimes\Gamma)\Gamma](\bm \omega)}=0,
\end{equation}
if $N$, $M$ and $P$ are symmetrel. Now we use the fact that, in characteristic zero, $\Cal H^\Omega$ is isomorphic to a symmetric algebra, namely the symmetric algebra of the free Lie algebra on $V$, where $V$ is the linear span of $\Omega$. A character of $\Cal H^\omega$ (i.e.~a symmetrel mould) is nothing but a point of the dual vector space $V^*$, and $\Cal H^\Omega$ is the algebra of polynomial functions on $V^*$. Hence for any $\bm \omega \in \Cal H^\omega-\{0\}$ there exists a symmetrel mould which does not vanish on $\bm\omega$. The Hopf algebra $(\Cal H^\Omega)^{\otimes 3}$ is the algebra of polynomial functions on $(V^*)^3$. From \eqref{comp-assoc} we have that any character of $(\Cal H^\Omega)^{\otimes 3}$ vanishes on $[(\Gamma\otimes\mop{Id})\Gamma-(\mop{Id}\otimes\Gamma)\Gamma](\bm \omega)$ for any $\bm\omega\in \Omega^*$, hence $(\Gamma\otimes\mop{Id})\Gamma=(\mop{Id}\otimes\Gamma)\Gamma$.\\

\noindent Conversely, dualizing $\Gamma$ gives a new associative ``mould composition" defined by:
\begin{equation}
	(M\diamond N)^{\bm\omega}=\sum_{s\ge 1}\sum_{\bm{\omega}=\bm{\omega}^1 .\,\cdots .\,\bm{\omega}^s}
	M^{\|\bm\omega^1\|\cdots\|\bm\omega^s\|}N^{\bm\omega^1\sqshu\cdots\ \sqshu\bm\omega^s}.\label{mcbis}
\end{equation}
From coassociativity of $\Gamma$ we can infer the associativity of $\diamond$, but nothing for $\circ$, although both compositions $\circ$ and $\diamond$ coincide when the right factor is symmetrel. The associativity of mould composition $\circ$ is then a stronger phenomenon than the coassociativity of the internal coproduct.
\end{rmk}

\begin{rmk}\rm
The compatibility of $\Gamma$ with the quasi-shuffle product immediately implies that the mould composition of two symmetrel moulds is symmetrel. Indeed, the composition $\circ$ of two symmetrel moulds coincides with their composition $\diamond$, which is their convolution product with respect to the internal coproduct $\Gamma$.
\end{rmk}

\begin{rmk}\rm
Observe that letters $\omega_i \in \Omega$ are group-like for the coproduct $\Gamma$. Indeed, since $\|\omega_i\|=\omega_i\in \Omega$ we have
$$
	\Gamma(\omega_i) = \omega_i \otimes \omega_i,
$$  
and equality \eqref{group-like} implies that the quasi-shuffle product of group-like elements is group-like:
$$
	\Gamma(\omega_i \qshu \omega_j)=(\omega_i\qshu \omega_j)\otimes(\omega_i \qshu \omega_j).
$$
Next, we consider the coproduct of words of length two
$$
	\Gamma(\omega_i\omega_j) = \omega_i\omega_j \otimes \omega_i \qshu \omega_j 
	+ \|\omega_i\omega_j\| \otimes \omega_i\omega_j.
$$  
Both $\omega_i \qshu \omega_j$ and the letter $\|\omega_i\omega_j\|$ are group-like elements. Therefore we consider words of length two to be {\sl{quasi-primitive}}. 
\end{rmk}

\begin{rmk}\rm
The coproduct $\Gamma$ is \textsl{internal} in the sense that we have:
\begin{equation}
	\Gamma(\Cal H^{\Omega}_{(\omega)})\subset\Cal H^{\Omega}_{(\omega)}\otimes\Cal H^{\Omega}_{(\omega)}
\end{equation}
for any letter $\omega\in\Omega$, where $\Cal H^{\Omega}_{(\omega)}\subset\Cal H^{\Omega}$ is the linear span of words of weight $\omega$.
\end{rmk}

The bialgebra $\Cal H^{\Omega}$ is pointed \cite{radford}. Its coradical $\Cal H^{\Omega}_0$ is the subalgebra of $\Cal H^{\Omega}$ (with respect to the quasi-shuffle product $\,\qshu$) generated by the letters $\omega\in\Omega$. It is also the linear span of $G$, where $G$ is the commutative monoid of nonzero group-like elements in $\Cal H^{\Omega}$. Let $\wt G$ be the abelian group associated with $G$, in which $G$ embeds canonically. Namely,
$$
	\wt G=G\times G/\sim,
$$
where $(g,h)\sim(g',h')$ if and only if there exists $k\in G$ with $g'=g\qshu k$ and $h'=h\qshu k$. The embedding $\iota:G\to\wt G$ is given by $\iota(g)=(g,\un)$. The product in $\wt G$ is induced by the diagonal product in $G\times G$, the unit is given by the class of $(g,g)$ for any $g\in G$, and the inverse of the class of $(g,h)$ is given by the class of $(h,g)$.\\

\noindent Let $\wt {\Cal H}^{\Omega}_0$ be the linear span of $\wt G$, and let $\wt {\Cal H}^{\Omega}$ be the bialgebra defined by:
$$
	\wt{\Cal H}^{\Omega}:=\Cal H^{\Omega}\otimes_{\Cal H^{\Omega}_0}\wt{\Cal H}^{\Omega}_0.
$$
Following \cite[Lemma 7.6.2]{radford} we deduce that $\wt{\Cal H}^{\Omega}$ is a Hopf algebra, obtained from $\Cal H^{\Omega}$ by formally inverting all group-like elements of $\Cal H^{\Omega}$. The antipode of $\wt{\Cal H}^{\Omega}$ can be computed as follows from its defining equations $S \star \mop{Id} = \un \varepsilon = \mop{Id} \star S$: together with $S(\omega)=\omega^{-1}$ for any letter $\omega\in\Omega\subset G\subset\wt G$, this implies for words of length two
$$
	S(\omega'\omega'') \qshu \omega' \qshu \omega'' + S(\|\omega'\omega''\|) \qshu \omega' \omega'' = 0, 
$$
from which we deduce that 
$$
	S(\omega'\omega'') = - {\omega'}^{-1} \qshu {\omega''}^{-1} \qshu \|\omega'\omega''\|^{-1} \qshu \omega' \omega''  .
$$
For words of length three we obtain
\allowdisplaybreaks
\begin{align*}
	S(\omega'\omega''\omega''') 
	&= -{\omega' }^{-1}\qshu {\omega''}^{-1}\qshu {\omega'''}^{-1}
			{\,\qshu}\Big( S(\|\omega'\omega''\|\omega''') \qshu \omega'\omega'' \qshu \omega'''
			+ S(\|\omega'\omega''\|\omega''') \qshu \omega'\qshu \omega''\omega''' \\
	&\qquad	+ S(\|\omega'\omega''\omega'''\|) \qshu \omega'\omega''\omega''' \Big) \\
	&= {\omega' }^{-1}\qshu {\omega''}^{-1}\qshu {\omega'''}^{-1}\qshu \|\omega'\omega''\omega'''\|^{-1}
			{\,\qshu}\Big({\|\omega'\omega''\|}^{-1}\qshu \|\omega'\omega'' \|\omega''' \qshu \omega'\omega'' \\
	&\qquad	+ {\|\omega''\omega'''\|}^{-1} \qshu \omega' \|\omega''\omega'''\|\qshu \omega''\omega'''
			+ \omega'\omega''\omega'''\Big).
\end{align*}

\section{Another approach via weak quasi-shuffles}
\label{sect:wqsh}

In the following we introduce the notion of weak quasi-shuffles. We then present an alternative, more pedestrian, proof of Theorem \ref{coprod-c}.

\subsection{Weak quasi-shuffles}
\label{ssect:weak}

We denote by $\mop{wqsh}(p,q;r)$ the set of \textsl{weak $(p,q)$-quasi-shuffles of type $r$}, i.e., surjective maps:
$$
	\sigma:\{1,\ldots,p+q\} \fleche 8\hskip -5.8mm\fleche 6 \{1,\ldots,p+q-r\}
$$
subject to the conditions $\sigma_1 \le \cdots \le \sigma_p$ and $\sigma_{p+1} \le \cdots \le \sigma_{p+q}$. We will denote by $\mop{qsh}(p,q)$ the set of all $(p,q)$-quasi-shuffles of any type, and $\mop{wqsh}(p,q)$ accordingly for weak quasi-shuffles.

\begin{lem}\label{wqsh}
Let $p$ and $q$ be two non-negative integers. Any weak $(p,q)$-quasi-shuffle $\varphi:\{1,\ldots,p+q\}\to\hskip -9.3pt\to \{1,\ldots,s\}$ (of type $p+q-s)$ admits a unique decomposition:
$$
	\varphi=\delta\circ\sigma,
$$
where:
\begin{itemize}
\item
	$\sigma$ is a nondecreasing surjection: $\{1,\ldots,p+q\}\to\hskip -9.3pt\to \{1,\ldots,t\}$ 
	for some $t\le p+q$\\ with $t\ge 2$, such that $\sigma_p<\sigma_{p+1}$,\\[-0.3cm]
\item
	$\delta:\{1,\ldots,t\}\to\hskip -9.3pt\to \{1,\ldots,s\}$ is a $(\sigma_p,t-\sigma_p)$-quasi-shuffle of type $r=t-s$.
\end{itemize}
\end{lem}

\begin{proof}
The nondecreasing surjection $\sigma$ and the quasi-shuffle $\delta$ are defined as follows: let $t$ be the sum $t'+t''$ where $t'$ (resp.~$t''$) stands for the number of values of the restriction of $\varphi$ to $\{1,\ldots,p\}$, resp.~$\{p+1,\ldots,p+q\}$. Let us denote by $\delta_1<\cdots<\delta_{t'}$, resp.~$\delta_{t'+1}<\cdots<\delta_{t''}$, the values reached by $\varphi_1,\ldots,\varphi_p$, resp.~$\varphi_{p+1},\ldots,\varphi_{p+q}$. The surjection $\delta:\{1,\ldots,t'+t''\}\to\hskip -9.3pt\to \{1,\ldots,s\}$ thus defined is a quasi-shuffle by definition. Then $\sigma_j:=k$ if $\varphi_j=\delta_k$ for $j\in\{1,\ldots,p\}$, and $\sigma_j:=t'+k$ if $\varphi_j=\delta_k$ for $j\in\{p+1,\ldots,p+q\}$. The decomposition is manifestly unique.
\end{proof}

\begin{rmk}
{\rm{
A \textsl{packed word} is a word on the alphabet $\mathbb N$ such that the set of letters appearing in the word is exactly $\{1,\ldots,s\}$ for some $s\in\mathbb N$. The set of surjections from $\{1,\ldots,n\}$ onto $\{1,\ldots,s\}$ is in canonical bijection with packed words of $n$ letters on the alphabet $\{1,\ldots,s\}$: for example the packed word $13224$ stands for the surjection from $\{1,2,3,4,5\}$ onto $\{1,2,3,4\}$ which sends $1$ to $1$, $2$ to $3$, $3$ to $2$, $4$ to $2$ and $5$ to $4$. The standardization of any word $w$ (packed or not) is the unique permutation $\sigma$ such that $w_i< w_j$ or ($w_i=w_j$ and $i<j$) implies $\sigma_i<\sigma_j$. For example, $\mop{Std}(13224)=14235$. With this interpretation at hand, given a weak $(p,q)$-quasi-shuffle $\varphi=\delta\circ\sigma$, the packed word of $\delta$ is obtained by erasing the repetitions of letters in $\varphi_1\ldots\varphi_p$ and $\varphi_{p+1}\ldots\varphi_{p+q}$. The packed word of $\sigma$ is obtained from $\varphi_1\ldots\varphi_{p+q}$ by packing both blocks $\varphi_1\ldots\varphi_p$ and $\varphi_{p+1}\ldots\varphi_{p+q}$, followed by shifting the second block by the maximum of the first. For example, for $\varphi=1224\big\vert 113$ we have $\delta=124\big\vert 13$ and $\sigma=1223\big\vert445$.}}
\end{rmk}

\noindent Now let $\varphi:\{1,\ldots,p+q\}\to\hskip -9.3pt\to\{1,\ldots,s\}$ be a weak quasi-shuffle (of type $p+q-s$). Let us consider the following set of $(p,q)$-quasi-shuffles associated to $\varphi$:
\begin{eqnarray*}
	\mop{qsh}\nolimits_\varphi(p,q)
	:=\big\{\eta\in\mop{qsh}(p,q),\, 
		&1) &\varphi(a)<\varphi(b)\Rightarrow \eta(a)<\eta(b)\hbox{ for any }a,b\in\{1,\ldots,p+q\},\\
		&2) &\varphi \hbox{ factorizes through }\eta\big\}.
\end{eqnarray*}

\begin{prop}\label{wqsh-bis}
Let $\varphi:\{1,\ldots,p+q\}\to\hskip -9.3pt\to\{1,\ldots,s\}$ be a weak $(p,q)$-quasi-shuffle (of type $p+q-s$), and let $\eta:\{1,\ldots,p+q\}\to\hskip -9.3pt\to\{1,\ldots,t'\}$ be a quasi-shuffle (of type $p+q-t'$) in $\mop{qsh}\nolimits_\varphi(p,q)$. There exists a unique non-decreasing surjection $\sigma[\eta]:\{1,\ldots,t'\}\to\hskip -9.3pt\to\{1,\ldots,s\}$ such that $\varphi=\sigma[\eta]\circ\eta$, and any factorization of $\varphi$ as a composition of a nondecreasing surjection with a quasi-shuffle (in that order) arises this way.
\end{prop}

\begin{proof}
The weak quasi-shuffle $\varphi$ factorizes through any $\eta\in\mop{qsh}_\varphi(p,q)$. The unique surjection $\sigma[\eta]:\{1,\ldots,t'\}\to\hskip -9.3pt\to\{1,\ldots,s\}$ thus defined is obviously nondecreasing if and only if the order condition 1) is verified.
\end{proof}
For example, for $\varphi=1224\big\vert 113$ we display the elements $\eta\in\mop{qsh}\nolimits_\varphi(4,3)$ and the corresponding nondecreasing surjections $\sigma[\eta]$:
\begin{align*}
&\eta & 1457\big\vert 236 && 2457\big\vert 136 && 3457\big\vert 126 && 1346\big\vert 125 && 2346\big\vert 125\\
&\sigma[\eta] & 1112234 && 1112234 && 1112234 && 112234 && 112234
\end{align*}

Let us describe all possible factorizations on the more elaborate example $\varphi=1224\big\vert 112334$, where we have $p=4$, $q=6$ and $s=4$: the standardization of $\varphi$ is $\mop{Std}\varphi=1459\big\vert 23678A$, where $A$ stands for $10$. Any element of $\mop{qsh}_\varphi(4,6)$ is obtained from $\mop{Std}\varphi$ by quasi-shuffling each preimage and concatenating, namely
\begin{itemize}

\item display $\mop{Std}\varphi$ according to the $\varphi$-preimages:
$$
	1\big\vert 23\hskip 12mm 45\big\vert 6\hskip 12mm \big\vert 78\hskip 12mm 9\big\vert A.
$$

\item Choose $s=4$ quasi-shuffles in $\mop{qsh}(1,2)$, $\mop{qsh}(2,1)$, $\mop{qsh}(0,2)$ and $\mop{qsh}(1,1)$ respectively, for example $\eta_1=2\big\vert 12$, $\eta_2=12\big\vert 1$, $\eta_3=\big\vert 12$ and $\eta_4=2\big\vert 1$. Note there is of course no freedom for choosing $\eta_3$.

\item Concatenate the $\eta_i$'s, which gives $\eta=2\, 34\ 8\big\vert 12\, 3\, 56\, 7$ here.

\end{itemize}

The nondecreasing surjection $\sigma[\eta]$ is then given by:
$$
	\sigma[\eta]=\underbrace{1\cdots 1}_{c_1}\underbrace{2\cdots 2}_{c_2}\cdots \underbrace{s\cdots s}_{c_s}\, ,
$$
where $c_j$ is the cardinality of the image of the quasi-shuffle $\eta_j$. In the example above, $\sigma[\eta]=11223344$. The cardinality of $\mop{qsh}_\varphi(4,6)$ for $\varphi=1224\big\vert 112334$ is equal to $5\times 5\times 1\times 3=75$. \\

\noindent Lemma \ref{wqsh} and Proposition \ref{wqsh-bis} can be visualized by the following diagram:
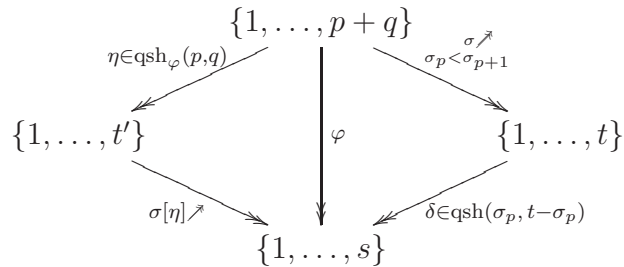
\diagramme{
\xymatrix{
&\{1,\ldots,p+q\}\ar@{->>}[dl]_{\eta\in\smop{qsh}_\varphi(p,q)}\ar@{->>}[dr]^{\hskip 3mm\sigma\,\sctos \atop \sigma_p<\sigma_{p+1}}\ar@{->>}[dd]^{\varphi} &\\
\{1,\ldots,t'\}\ar@{->>}[dr]_{\hskip -3mm\sigma[\eta]\,\sctos} && \{1,\ldots, t\}\ar@{->>}[dl]^{\hskip 4mm\delta\in\smop{qsh}(\sigma_p,\,t-\sigma_p)}\\
&\{1,\ldots,s\}&
}
}
\noindent Here $\scto$ indicates a nondecreasing surjection. The right wing displays the unique factorization in  Lemma \ref{wqsh} and left wing represents the factorization in Proposition \ref{wqsh-bis}.

\subsection{A pedestrian proof of Theorem \ref{coprod-c} via weak quasi-shuffles}
\label{sect:main}

Let us check the compatibility of the internal coproduct $\Gamma$ with quasi-shuffle for two length one words,  i.e., for the quasi-shuffle product of two letters $\omega_1,\omega_2 \in \Omega$:
\allowdisplaybreaks
\begin{eqnarray}
	\Gamma(\omega_1\qshu \omega_2)
	&=&\Gamma(\omega_1\omega_2+\omega_2\omega_1+[\omega_1+\omega_2]) \nonumber\\
	&=&\omega_1\omega_2\otimes(\omega_1\qshu \omega_2) +[\omega_1+\omega_2]\otimes \omega_1\omega_2
		+\omega_2\omega_1\otimes (\omega_2\qshu \omega_1) 
		+[\omega_2+\omega_1]\otimes \omega_2\omega_1 \nonumber\\
	&&\quad	+[\omega_1+\omega_2]\otimes[\omega_1+\omega_2] \nonumber\\
	&=&(\omega_1\qshu \omega_2)\otimes(\omega_1 \qshu \omega_2) \label{group-like}\\
	&=&\Gamma(\omega_1)\qshu\Gamma(\omega_2). \nonumber
\end{eqnarray}
The reader is invited to check the less obvious case of two words $\omega_1\omega_2$ and $\omega_3$ of lengths two and one, respectively. The general case could certainly be handled by induction on the sum of the lengths of the two words involved, but we give here a direct proof based on the notion of weak quasi-shuffle defined in Section \ref{sect:wqsh}. Let us introduce some more notations: for any word $\bm\omega=\omega_1\cdots\omega_n$ and for any surjection $\sigma:\{1,\ldots, n\} \to \hskip -9.3pt\to \{1,\ldots,s\}$ with $s \le n$, we will denote by $\bm\omega^\sigma$ the word $\omega^\sigma_1\cdots\omega^\sigma_s$ with:
$$
	\omega^\sigma_k:=\Big[\sum_{\sigma_j=k}\omega_j\Big],
$$
thus extending to any surjection the notation introduced for quasi-shuffles in the beginning of this section. If $\tau:\{1,\ldots,s\}\to\hskip -9.3pt\to\{1,\ldots,p\}$ is another surjection, we obviously have:
\begin{equation}\label{surj-comp}
	(\bm\omega^\sigma)^\tau=\bm\omega^{\tau\circ\sigma}.
\end{equation}
We also introduce
$$
	\bm\omega_\sigma^k := \omega_{j_1}\cdots\omega_{j_{r(k)}},
$$ 
where $j_1,\ldots,j_{r(k)}$ are the $\sigma$-preimages of $k$ arranged in increasing order. Hence we have:
\begin{equation}
	\bm\omega^\sigma=||\bm\omega_\sigma^1||\cdots||\bm\omega_\sigma^s||.
\end{equation}
\ignore{For any pair $\sigma,\tau$ of surjective maps of domain $\{1,\ldots,n\}$, we say that $\tau$ is a \textsl{$\sigma$-quasi-shuffle} if $\tau$ is (strictly) increasing on each block $\sigma^{-1}(k)$ for any $k$ in the range of $\sigma$.} We are now ready to give another expression for the decomposition coproduct of a word $\bm\omega=\omega_1\cdots\omega_n$:
\begin{equation}
	\Gamma(\bm\omega)
	=\sum_{s\ge 1}\sum_{\sigma:\{1,\ldots,n\}\scto\{1,\ldots,s\}} 
	\bm\omega^\sigma\otimes(\bm\omega_\sigma^1\qshu\cdots\qshu\bm\omega_\sigma^s).
	\ignore{&=&\sum_{s\ge 1}\sum_{\sigma:\{1,\ldots,n\}\scto\{1,\ldots,s\}}
	\sum_{\tau\in\smop{qsh}(\sigma)} \bm\omega^\sigma\otimes\bm\omega^\tau}
\end{equation}
Recall that $\scto$ indicates a nondecreasing surjection. Let us now consider two words $\bm\omega'=\omega_1\cdots\omega_p$ and $\bm\omega''=\omega_{p+1}\cdots\omega_{p+q}$, and compute (with $\bm\omega:=\bm\omega'.\bm\omega''$):
\begin{eqnarray*}
	\Gamma(\bm\omega'\qshu\bm\omega'')
	&=&\sum_{r\ge 0}\sum_{\eta\in\smop{qsh}(p,q;r)} \Gamma(\bm\omega^\eta)\\
	&=&\sum_{r\ge 0}\sum_{\eta\in\smop{qsh}(p,q;r)} \sum_{s\ge 1}\sum_{\wt\sigma:\{1,\ldots,p+q-r\}\scto\{1,\ldots,s\}}
	\bm\omega^{\wt\sigma\circ\eta}\otimes\big((\bm\omega^\eta)_{\wt\sigma}^1
	\qshu\cdots\qshu(\bm\omega^\eta)_{\wt\sigma}^s\big).
\end{eqnarray*}
Now remark that the surjections $\wt\sigma\circ\eta$ above are weak quasi-shuffles of type $p+q-s$. Hence we can gather the terms corresponding to the same weak quasi-shuffle, which yields, according to Proposition \ref{wqsh-bis}: 
\begin{equation}
	\Gamma(\bm\omega'\qshu\bm\omega'')=
	\sum_{s\ge 1}\sum_{\varphi\in\smop{wqsh}(p,q; p+q-s)}\bm\omega^{\varphi}
	\otimes\left(\sum_{r\ge 0}\sum_{\eta\in\smop{qsh}_\varphi(p,q;r)}
	(\bm\omega^\eta)_{\sigma[\eta]}^1\qshu\cdots\qshu(\bm\omega^\eta)_{\sigma[\eta]}^s\right).
\end{equation}
On the other hand, using Lemma \ref{wqsh} we get:
\allowdisplaybreaks
\begin{eqnarray*}
	\lefteqn{\Gamma(\bm\omega')\qshu\Gamma(\bm\omega'')}\\
	&=& \sum_{t'\ge 1 \atop t''\ge 1}\sum_{\sigma':\{1,\ldots,p\}\sctos\{1,\ldots,t'\} \atop \sigma'':\{p+1,\ldots,p+q\}\sctos\{1,\ldots,t''\}}{\ignore{\sum_{t''\ge 1}
	\sum_{\sigma'':\{p+1,\ldots,p+q\}\scto\{1,\ldots,t''\}}}}({\bm\omega'}^{\sigma'}\qshu{\bm\omega''}^{\sigma''})\otimes
	({\bm\omega'}_{\sigma'}^1\qshu\cdots\qshu{\bm\omega'}_{\sigma'}^{t'}\qshu{\bm\omega''}_{\sigma''}^1\qshu\cdots	\qshu{\bm\omega''}_{\sigma''}^{t''})\\
	&=& \sum_{t\ge 2} \sum_{\sigma:\{1,\ldots,p+q\}\sctos\{1,\ldots,t\} \atop \sigma_{p}<\sigma_{p+1}}
	\sum_{r\ge 0}\sum_{\delta\in\smop{qsh}(\sigma_p,t-\sigma_p;r)}
	\bm\omega^{\delta\circ\sigma}\otimes({\bm\omega}_{\sigma}^1\qshu\cdots\qshu{\bm\omega}_{\sigma}^{t})\\
	&=& \sum_{s\ge 1}\sum_{\varphi\in\smop{wqsh}(p,q;p+q- s)}\bm\omega^{\varphi}\otimes({\bm\omega}_{\sigma}^1
	\qshu\cdots\qshu{\bm\omega}_{\sigma}^{t}),
\end{eqnarray*}
where $\sigma:\{1,\ldots,p+q\}\! \scto\ \{1,\ldots,t\}$ is the increasing component of $\varphi$ in the decomposition given by Lemma \ref{wqsh}. Hence, compatibility of the internal coproduct $\Gamma$ with the quasi-shuffle product will immediately stem from the following lemma:

\begin{lem}\label{paquets}
For any weak $(p,q)$-quasi-shuffle $\varphi=\delta\circ\sigma:\{1,\ldots,p+q\}\to\hskip -9.3pt\to\{1,\ldots,s\}$, where $\sigma:\{1,\ldots,p+q\}\cto \{1,\ldots,t\}$ and $\delta: \{1,\ldots,t\}\to\hskip -9.3pt\to \{1,\ldots,s\}$ are the two components of $\varphi$ given by Lemma \ref{wqsh}, we have:\rm
\begin{equation}\label{intermediaire}
	\sum_{\eta\in\smop{qsh}_{\varphi}(p,q)}
	(\bm\omega^\eta)_{\sigma[\eta]}^1\qshu\cdots\qshu(\bm\omega^\eta)_{\sigma[\eta]}^s
	={\bm\omega}_{\sigma}^1 \qshu\cdots\qshu{\bm\omega}_{\sigma}^{t}.
\end{equation}
\end{lem}

\begin{proof}
The right-hand side of \eqref{intermediaire} can also be written as follows:
\begin{equation}\label{rhs-intermediaire}
	R={\bm\omega}_{\sigma}^1\qshu\cdots\qshu{\bm\omega}_{\sigma}^{t}=A_1\qshu\cdots\qshu A_s,
\end{equation}
where, for any $j\in\{1,\ldots,s\}$, the $j^{\smop{th}}$ term in the product above in given by:
\begin{equation}\label{aj}
	A_j=\ \mathop{\scalebox{1.6}{$\ \qshu$}}\limits_{\delta_\ell=j}\ \bm\omega_\sigma^\ell. 
\end{equation}
It is equal to one single word $\bm\omega_\sigma^{\ell}$ or to the quasi-shuffle $\bm\omega_\sigma^{\ell}\qshu\bm\omega_\sigma^m$, according to whether $\delta^{-1}(j)$ contains one element $\ell$ or two elements $\ell< m$. In the second case, $A_j$ is the sum of all the words obtained by quasi-shuffling the letters of $\bm\omega_\sigma^{\ell}$ with those of $\bm\omega_\sigma^{m}$. We have $\ell\le\sigma_p$ and $m\ge \sigma_{p+1}$, hence the letters of $\bm\omega_\sigma^{\ell}$ are of rank $\le p$, whereas those of $\bm\omega_\sigma^{m}$ are of rank $>p$. Plugging \eqref{aj} inside \eqref{rhs-intermediaire}, we can thus see, in the light of the description of $\mop{qsh}_\varphi(p,q)$ which follows the proof of Proposition \ref{wqsh-bis}, that we have:
\begin{equation}
	{\bm\omega}_{\sigma}^1\qshu\cdots\qshu{\bm\omega}_{\sigma}^{t}
	=\sum_{\eta\in\smop{qsh}_\varphi(p,q)}(\bm\omega^\eta)_{\sigma[\eta]}^1
	\qshu\cdots\qshu(\bm\omega^\eta)_{\sigma[\eta]}^s,
\end{equation}
which proves Lemma \ref{paquets}.
\end{proof}

\noindent This in turn proves the first assertion of Theorem  \ref{coprod-c}. In order to prove coassociativity, let us compute for a word $\bm\omega$ of length $n$:
\allowdisplaybreaks
\begin{eqnarray*}
	\lefteqn{(\Gamma\otimes \mop{Id})\Gamma(\bm\omega)
	=\sum_{s\ge 1}\sum_{\sigma:\{1,\ldots,n\}\scto\{1,\ldots,s\}}(\Gamma\otimes \mop{Id})\big(\bm\omega^\sigma
	\otimes(\bm\omega_\sigma^1\qshu\cdots\qshu\bm\omega_\sigma^s)\big)}\\
	&=&\sum_{s\ge r\ge 1}\sum_{\sigma:\{1,\ldots,n\}\scto\{1,\ldots,s\}}\sum_{\tau:\{1,\ldots,s\}\scto\{1,\ldots,r\}}
	\hskip -6mm\bm\omega^{\tau\circ\sigma}\otimes \big((\bm\omega^\sigma)_\tau^1\qshu\ldots
	\qshu(\bm\omega^\sigma)_\tau^r\big)\otimes(\bm\omega_\sigma^1\qshu\cdots\qshu\bm\omega_\sigma^s).
\end{eqnarray*}
On the other hand, using the compatibility of $\Gamma$ with the quasi-shuffle product, we have:
\allowdisplaybreaks
\begin{eqnarray*}
	\lefteqn{(\mop{Id}\otimes\Gamma)\Gamma(\bm\omega)
	=\sum_{r\ge 1}\sum_{\rho:\{1,\ldots,n\}\scto\{1,\ldots,r\}}(\mop{Id}\otimes\Gamma)\big(\bm\omega^\rho
	\otimes(\bm\omega_\rho^1\qshu\cdots\qshu\bm\omega_\rho^r)\big)}\\
	&=&\sum_{r\ge 1}\sum_{\rho:\{1,\ldots,n\}\scto\{1,\ldots,r\}}\sum_{m_1,\ldots,m_r\ge 1}\ 
	\sum_{\sigma_i:\rho^{-1}(i)\scto {\{1,\ldots,m_i\} \atop i=1,\ldots,r}}
	\hskip -10mm\bm\omega^\rho\otimes\big((\bm\omega_\rho^1)^{\sigma_1}\qshu\cdots\qshu
	(\bm\omega_\rho^r)^{\sigma_r}\big)\\
	&&\hskip 45mm \otimes (\bm\omega_\rho^1)_{\sigma_1}^1\qshu
	\cdots\qshu(\bm\omega_\rho^1)_{\sigma_1}^{m_1}\ \qshu\ \cdots\ \qshu\ 
	(\bm\omega_\rho^r)_{\sigma_r}^1\qshu\cdots\qshu(\bm\omega_\rho^r)_{\sigma_r}^{m_r}.
\end{eqnarray*}
Now, for any $s$ such that $r\le s\le n$, choosing $m_1,\ldots,m_r$ with $m_1+\cdots+m_r=s$ amounts to choosing a nondecreasing surjection $\tau:\{1,\ldots , s\}\cto\{1,\ldots,r\}$, and the surjections $(\sigma_i)_{i=1,\ldots,r}$ concatenate together to give a nondecreasing surjection $\sigma:\{1,\ldots,n\}\cto\{1,\ldots,s\}$ such that $\rho=\tau\circ\sigma$. Conversely, any pair $(\sigma,\tau)$ arises this way, the $\sigma_i$'s being the restriction of $\sigma$ to the preimages $\rho^{-1}(i),\,i=1,\ldots,r$. Hence we get:
\goodbreak
\allowdisplaybreaks
\begin{eqnarray*}
	(\Gamma\otimes\mop{Id})\Gamma(\bm\omega)
	&=&\sum_{s\ge r\ge 1}
	\sum_{\sigma:\{1,\ldots,n\}\scto\{1,\ldots,s\}}
	\sum_{\tau:\{1,\ldots,s\}\scto\{1,\ldots,r\}}
	\bm\omega^{\tau\circ\sigma}
	\otimes \big((\bm\omega^\sigma)_\tau^1 \qshu\ldots\qshu(\bm\omega^\sigma)_\tau^r\big)\\ 
	& & \hskip 75mm  \otimes(\bm\omega_\sigma^1
	\qshu\cdots\qshu\bm\omega_\sigma^s)\\
	&=&(\mop{Id}\otimes\Gamma)\Gamma(\bm\omega).
\end{eqnarray*}
It remains to check the comodule-Hopf algebra property. On the one hand we have for any word $\bm\omega$ of length $n$:
\begin{equation*}
	(\Delta\otimes\mop{Id})\Gamma(\bm\omega)
	=\sum_{s\ge 1}\sum_{\sigma:\{1,\ldots,n\}\scto\{1,\ldots, s\}}\sum_{\bm\omega^\sigma
	=\bm u.\bm v}\bm u\otimes\bm v\otimes (\bm\omega_\sigma^1\qshu\cdots\,\qshu\bm\omega_\sigma^s).
\end{equation*}
On the other hand, we compute:
\allowdisplaybreaks
\begin{eqnarray*}
	(\mop{Id}\otimes\mop{Id}\otimes\,\qshu)\tau_{23}(\Gamma\otimes\Gamma)\Delta(\bm\omega)
	&=&\sum_{\bm\omega=\bm\omega^1.\bm\omega^2}(\mop{Id}\otimes\mop{Id}\otimes\,\qshu)
	\tau_{23}(\Gamma\otimes\Gamma)(\bm\omega^1\otimes\bm\omega^2)\\
	&&\hskip -65mm =\sum_{\bm\omega=\bm\omega^1.\bm\omega^2}
	\sum_{s_1,s_2\ge 1}\sum_{\scriptstyle{\sigma_1:\{1,\ldots,p\}\scto\{1,\ldots, s_1\}}\atop \scriptstyle\sigma_2:\{1,\ldots,q\}\scto\{1,\ldots, s_2\}}(\bm\omega^1)^{\sigma_1}\otimes(\bm\omega^2)^{\sigma_2}
	\otimes \big((\bm\omega^1)_{\sigma_1}^1\,\qshu\cdots\,\qshu(\bm\omega^1)_{\sigma_1}^{s_1} \,
	\qshu (\bm\omega^2)_{\sigma_2}^1\,\qshu\cdots\,\qshu(\bm\omega^2)_{\sigma_2}^{s_2}\big)
\end{eqnarray*}
with $p=|\bm\omega^1|$ and $q=|\bm\omega^2|$ (recall that $|-|$ stands for the length of a word). The two surjections $\sigma_1$ and $\sigma_2$ concatenate to give rise to a surjection $\sigma:\{1,\ldots,n\}\cto\{1,\ldots,s_1+s_2\}$ with $\sigma_p<\sigma_{p+1}$, hence both expressions match. Checking commutativity of the two other diagrams is more easy and left to the reader.

\section{Contracting arborification}
\label{arborification}

The notion of arborification has been introduced by J.~Ecalle in \cite{E92}. A detailed presentation can be found in \cite{EV}. See also \cite{FM} for a Hopf-algebraic presentation. An \textsl{arborescent mould} is a collection $(M^F)$, where $F$ is any rooted forest decorated by the alphabet $\Omega$. \textsl{Arborification}, in its simple or contracting version, associates an arborescent mould to any ordinary mould in the sense of Section \ref{sect:moules}. We shall only consider the contracting version here.

\subsection{Decorated rooted forests}
\label{ssect:decTree}

A rooted forest is a finite oriented graph $F$ without cycles, such that every vertex has at most one incoming edge. Vertices with no incoming edge are called the roots, and it is easily seen that any nonempty forest has at least one root. A rooted forest  with one single root is a \textsl{rooted tree}. The set $\Cal V(F)$ of vertices of $F$ is partially ordered, i.e., $u\le v$ if and only if there exists an oriented path from one root to $v$ through $u$.\\

An $\Omega$-decorated rooted forest (where $\Omega$ is a given set) is a pair $F=(\overline F,d)$ where $\overline F$ is a rooted forest and where $d:\Cal V(\overline F)\to\Omega$ is the decoration. Let $\Cal H^{\Omega}_<$ be the linear span of rooted forests decorated by $\Omega$. It is the free commutative algebra over the linear span $\Cal T^\Omega$ of $\Omega$-decorated rooted trees. It is a graded commutative Hopf algebra \cite{ConnesKreimer,Dur}, with the following coproduct:
\begin{equation}\label{delta}
	\Delta(F)	=\sum_{V_1\sqcup V_2 =\Cal V(F) \atop V_1<V_2}F\restr{V_2}\otimes F\restr{V_1}.
\end{equation}
Here $\Cal V(F)$ stands for the set of vertices of $F$, the restriction of the forest $F$ to a subset of $\Cal V(F)$ is obtained by keeping only the edges joining two vertices in the subset, and $V_1<V_2$ means that for any $x\in V_1$ and $y\in V_2$, one has  $y\not < x$. Such a couple $(V_1,V_2)$ is called an \textsl{admissible cut} \cite{Mu}. For any $b\in\Omega$, the operator $B_+^b:\Cal H^{\Omega}_<\to\Cal H^{\Omega}_<$, defined by grafting all the trees of a forest on a common root decorated by $b$, verifies the following \textsl{cocycle equation}:
$$
	\Delta\big(B_+^b(F)\big)=(\mop{Id}\otimes B_+^b)\Delta(F)+B_+^b(F)\otimes\un,
$$
where $\un$ here stands for the empty forest. The operator $L^b:\Cal H^\Omega\to\Cal H^\Omega$ defined by $L^b(\bm\omega)=\bm\omega b$ verifies the same cocycle equation with respect to the deconcatenation coproduct:
$$
	\Delta\big(L^b(\bm\omega)\big)=(\mop{Id}\otimes L^b)\Delta(\bm\omega)+L^b(\bm \omega)\otimes\un.
$$
Now there is a unique surjective Hopf algebra morphism \cite{F}
$$
	\mathfrak a:\Cal H_<^\Omega\longrightarrow\hskip -5.8mm\longrightarrow \Cal H^\Omega,
$$
such that $\mathfrak a(\un)=\un$, and such that:
$$
	\mathfrak a\circ B^b_+=L^b\circ\mathfrak a,
$$
called \textsl{contracting arborification}\footnote{The \textsl{simple arborification} $\mathfrak a_0$ is obtained in a similar way, except that the quasi-shuffle Hopf algebra $\Cal H^\Omega$ is replaced by the shuffle Hopf algebra. For any forest $F$, its image $\mathfrak a_0(F)$ is the sum of all linear extensions of $F$, without contractions.}. For any forest $F$, its image $\mathfrak a(F)$ is the sum of all linear extensions of $F$, with contractions allowed. For example:
$$
	\mathfrak a(\arbrebbdec{\omega_3}{\omega_1}{\omega_2})
	=\omega_1\omega_2\omega_3+\omega_2\omega_1\omega_3+[\omega_1+\omega_2]\omega_3.
$$
For any mould $M: \Cal H^\Omega \to \Cal A$, the corresponding so-called \textsl{arborified mould} will be defined by 
$$
	M_<:=M\circ\mathfrak a:\Cal H^\Omega_<\to \Cal A.
$$

\subsection{The internal coproduct on decorated rooted forests}
\label{ssect:coprodDecTrees}

Let $F$ be a rooted forest decorated by $\Omega$. A \textsl{covering subforest} is a partition $G$ of the set $\Cal V(F)$ of vertices of $F$ into connected components, which yields another $\Omega$-decorated rooted forest still denoted by $G$. The \textsl{contracted forest} $F/G$ is obtained by shrinking each connected component of $G$ onto one vertex, which will be decorated by the sum of the decorations of all vertices of the connected component involved. The coproduct defined by:
\begin{equation}
\label{extraction-contraction}
	\Gamma(F)=\sum_{G\hbox{ \eightrm covering} \atop \hbox{ \eightrm  subforest of }F}F/G\otimes G
\end{equation}
is coassociative and compatible with the product and internal, in the sense that if a forest $F$ has total weight $\omega\in\Omega$, any covering subforest $G$ has also total weight $\omega$, as well as the contracted forest $F/G$. Moreover, $(\Cal H^\Omega_<,\cdot,\Delta)$ is a right comodule-Hopf algebra over the bialgebra $(\Cal H^\Omega_<,\cdot,\Gamma)$. This has been proved in \cite{CEM} without decorations, the adaptation to the decorated case is straightforward and left to the reader\footnote{Note that compared to \cite{CEM}, we have flipped the internal coproduct. As a result, we get a right comodule-bialgebra structure instead of a left one.}.

\begin{thm}\label{arbo-interne}
The contracting arborification respects the internal coproducts, namely:
\begin{equation}
	\Gamma\circ\mathfrak a=(\mathfrak a\otimes\mathfrak a)\circ\Gamma.
\end{equation}
\end{thm}

\begin{proof}
Let us introduce some more notations. For any $\Omega$-decorated forest $F$, we denote by $d_v$ the decoration of the vertex $v$ for any $v\in\Cal V(F)$. The notation $\sigma:\Cal V(F)\cto\,\{1,\ldots,s\}$ (resp.~$\sigma:\Cal V(F)\ccto\,\{1,\ldots,s\}$) stands for a surjective map from $\Cal V(F)$ onto $\{1,\ldots,s\}$ such that $\sigma_u\le\sigma_v$ (resp.~$\sigma_u<\sigma_v$) whenever $u<v$ for the partial order on $\Cal V(F)$ induced by the rooted forest. For any surjective map $\sigma:\Cal V(F)\to\hskip-9.3pt\to\{1,\ldots,s\}$, we denote by $F^\sigma$ the word $F_1^\sigma\cdots F_s^\sigma$, where:
$$
	F_j^\sigma:=\Big[\sum_{\sigma_v=j}d_v\Big]\in\Omega.
$$
With these notations at hand, the contracting arborification of the rooted decorated forest $F$ can be displayed as follows:
\begin{equation}
	\mathfrak a(F)=\sum_{s\ge 1}\sum_{\sigma:\Cal V(F)\sccto\{1,\ldots,s\}}F^\sigma.
\end{equation}
On the one hand we have:
\allowdisplaybreaks
\begin{eqnarray*}
	\Gamma\circ\mathfrak a(F)
	&=&\sum_{s\ge 1}\sum_{\sigma:\Cal V(F)\sccto\{1,\ldots,s\}}\Gamma(F^\sigma)\\	
	&=&\sum_{s\ge r\ge 1}\ \sum_{\sigma:\Cal V(F)\sccto\{1,\ldots,s\}}\ \sum_{\tau:\{1,\ldots,s\}\scto\{1,\ldots,r\}}
	F^{\tau\circ\sigma}\otimes\big((F^\sigma)_\tau^1\qshu\cdots\qshu(F^\sigma)_\tau^r\big)\\
	&=&\sum_{s\ge r\ge 1}\ \sum_{\sigma:\Cal V(F)\sccto\{1,\ldots,s\}}\ \sum_{\tau:\{1,\ldots,s\}\scto\{1,\ldots,r\}}\
	 \sum_{\rho\in\smop{qsh}(\tau)}F^{\tau\circ\sigma}\otimes F^{\rho\circ\sigma},
\end{eqnarray*}
where $\mop{qsh}(\tau)$ stands for surjective maps from $\{1,\ldots,s\}$ onto $\{1,\ldots, t\}$ (for some $t\le s$) which are increasing on each block $\tau^{-1}(j),\,j\in\{1,\ldots,r\}$. On the other hand,
\allowdisplaybreaks
\begin{eqnarray*}
	(\mathfrak a\otimes\mathfrak a)\circ\Gamma(F)
	&=&\sum_{ G\hbox{ \eightrm covering} \atop \hbox{ \eightrm  subforest of }F}\mathfrak a(F/G)\otimes\mathfrak a(G)\\
	&=&\sum_{G\hbox{ \eightrm covering} \atop \hbox{ \eightrm  subforest of }F}\ \sum_{r,t\ge 1}\ 
	\sum_{{\scriptstyle\alpha:\Cal V(F/G)\sccto\{1,\ldots,r\}\atop\scriptstyle\beta:\Cal V(G)\sccto \{1,\ldots, t\}}}
	(F/G)^\alpha\otimes G^\beta.
\end{eqnarray*}
Theorem \ref{arbo-interne} will then directly stem from the following lemma.

\begin{lem}\label{ab}
Let $F$ be a rooted forest decorated by $\Omega$. There is a bijective correspondence $\Phi$ from the set
$$	
	A:=\big\{(\sigma,\tau,\rho),\ \sigma:\Cal V(F)\ccto\{1,\ldots,s\},\, {\tau:\{1,\ldots, s\}\cto\{1,\ldots,r\} \atop
	\hbox{ and }\rho\in\mop{qsh}(\tau)\hbox{ for some }s\ge r\ge 1}\big\}
$$
onto the set 
\begin{eqnarray*}
	B:=\big\{(G,\alpha,\beta),\ G \hbox { covering subforest of }F, \, {\alpha:\Cal V(F/G)\ccto\{1,\ldots,r\} \atop \hbox{and}\
	\beta:\Cal V(G)\ccto\{1,\ldots,t\}
	\hbox{ for some }r,t\ge 1}\big\}
\end{eqnarray*}
such that, for $(G,\alpha,\beta)=\Phi(\sigma,\tau,\rho)$, the identities $F^{\tau\circ\sigma}=(F/G)^\alpha$ and $F^{\rho\circ\sigma}=G^\beta$ hold.
\end{lem}

\begin{proof}
Let $(\sigma,\tau,\rho)\in A$, with $\sigma:\Cal V(F)\ccto\{1,\ldots,s\}$, $\tau:\{1,\ldots, s\}\cto\{1,\ldots,r\}$ and $\rho\in\mop{qsh}(\tau)$  for some $s\ge r\ge 1$. The connected components of the blocks $(\tau\circ\sigma)^{-1}(j)$ (for $j\in\{1,\ldots,r\}$) define a covering subforest $G$ of $F$, and $\tau\circ\sigma:\Cal V(F)\ccto\{1,\ldots, r\}$ factorizes itself through a unique $\alpha:\Cal V(F/G)\ccto\{1,\ldots, r\}$. Moreover $\beta:=\rho\circ\sigma:\Cal V(G)\ccto\{1,\ldots,t\}$ for some $t\le s$, which defines $\Phi:A\to B$ by $\Phi(\sigma,\tau,\rho)=(G,\alpha,\beta)$. The identities $F^{\tau\circ\sigma}=(F/G)^\alpha$ and $F^{\rho\circ\sigma}=G^{\rho\circ\sigma}=G^\beta$ obviously hold.\\

Conversely, if $(G,\alpha,\beta)\in B$, where $G$ is a covering subforest of $F$ with $\alpha:\Cal V(F/G)\ccto\{1,\ldots,r\}$ and $\beta:\Cal V(G)\ccto\{1,\ldots,t\}$ for some $r,t\ge 1$, the surjection $\alpha$ lifts to a surjection $\sigma':\Cal V(F)\cto\{1,\ldots,r\}$ constant on the connected components of $G$. There is a unique $s\ge r$  and a unique surjection $\sigma:\Cal V(F)\ccto\{1,\ldots,s\}$ such that $\sigma_u\le \sigma_v$ if and only if $\sigma'_u\le\sigma'_v$: the \textsl {standardization} of $\sigma'$. Then there exists a unique $\tau:\{1,\ldots,s\}\cto\{1,\ldots,r\}$ such that $\sigma'=\tau\circ\sigma$, and a unique $\rho:\{1,\ldots,s\}\to\hskip -9.3pt\to\{1,\ldots,t\}$ in $\mop{qsh}(\tau)$ such that $\beta=\rho\circ\sigma$, which defines $\Phi^{-1}$.
\end{proof}

\noindent\textit{Proof of Theorem \ref{arbo-interne} (end).} According to Lemma \ref{ab}, the two terms $\Gamma\circ\mathfrak a(F)$ and $(\mathfrak a\otimes\mathfrak a)\circ\Gamma(F)$ match for any decorated rooted forest $F$.
\end{proof}

\subsection{Reminder on pre-Lie algebras and their enveloping algebras}
\label{sect:pre-lie}

A {\sl left pre-Lie algebra\/} over a field $\bm k$ is a $\bm k$-vector space $P$ with a bilinear binary composition $ \rhd  $ that satisfies the left pre-Lie identity:
\begin{equation}
    (a \rhd   b) \rhd   c-a \rhd  (b \rhd  c)=
    (b \rhd   a) \rhd   c-b \rhd  (a \rhd   c),
    \label{prelie}
\end{equation}
for $a,b,c \in P$. The left pre-Lie identity rewrites as:
\begin{equation}\label{prelie1-2}
	L_{[a,b]}=[L_a,L_b],
\end{equation}
where $L_a: P \to P$ is defined by $L_ab=a\rhd b$, and where the bracket on the left-hand side is defined by $[a,b]:=a\rhd b-b\rhd a$. As an easy consequence this bracket satisfies the Jacobi identity.\\

Let us recall an important result by D.~Guin and J-M.~Oudom \cite{GO, OudomGuin08, M10}. Let $P$ be any left pre-Lie algebra, and let $S(P)$ be its symmetric algebra, i.e., the free commutative algebra on $P$. Let $P_{\smop{Lie}}$ be the underlying Lie algebra of $P$, i.e., the vector space $P$ endowed with the Lie bracket given by $[a,b]=a\rhd b-b\rhd a$ for any $a,b\in P$, and let $\Cal U(P)$ be the enveloping algebra of  the Lie algebra $P_{\smop{Lie}}$, endowed with its usual increasing filtration. Let us consider the associative algebra $\Cal U(P)$ as a left module over itself. There exists a left $\Cal U(P)$-module structure on $S(P)$ and a canonical left $\Cal U(P)$-module isomorphism $\eta_P : \Cal U(P) \to \Cal
S(P)$, such that the associated graded  linear map $\mop{Gr}\eta_P:\mop{Gr}\Cal U(P)\to S(P)$ is an isomorphism of commutative graded algebras.\\

\noindent The proof in \cite[Paragraph 4.3]{M10} can be summarized as follows. The Lie algebra morphism 
\begin{align*}
	L:P	&\longrightarrow\mop{End}P\\
	a	&\longmapsto (L_a:b\mapsto a\rhd b)
\end{align*}
extends by Leibniz rule to a unique Lie algebra morphism $ L:P\to \mop{Der}S(P)$. It is easily seen that the map $m : P\to\mop{End}S(P)$ defined by:
\begin{equation}
	m_av=av+L_av
\end{equation}
is a Lie algebra morphism. Now $m$ extends, by universal property of the enveloping algebra, to a unique algebra morphism $m:\Cal U(P)\to\mop{End} S(P)$. The linear map:
\begin{eqnarray*}
	\eta_P:\Cal U(P)
	&\longrightarrow& S(P)\\
	u&\longmapsto & m_u.1
\end{eqnarray*}
is clearly a morphism of left $\Cal U(P)$-modules. It is immediately seen by induction that for any $a_1,\ldots,a_n\in P$ we have $\eta_P(a_1\cdots a_n)=a_1\cdots a_n+v$ where $v$ is a sum of terms of degree smaller or equal to $n-1$, which proves the result. Functorial properties are moreover fulfilled:

\begin{prop}
Let $P$ and $Q$ be two left pre-Lie algebras over the same field $\bm k$. Any pre-Lie morphism $\alpha:P\to Q$ uniquely extends to two algebra morphisms $\overline\alpha:S(P)\to S(Q)$ and $\wt\alpha:\Cal U(P)\to\Cal U(Q)$ such that the following diagram commutes:
\diagramme{
\xymatrix{\Cal U(P)\ar[r]_{\eta_P}^\sim\ar[d]_{\wt\alpha} &S(P)\ar[d]^{\overline\alpha}\\
\Cal U(Q)\ar[r]_{\eta_Q}^\sim & S(Q)
}
}
\end{prop}

\begin{proof}
By induction on the filtration degree, the degree zero case being trivial. For $a\in P$ and $u\in\Cal U(P)$ we have:
\begin{eqnarray*}
	\overline\alpha\circ \eta_P(au)
	&=&\overline\alpha(m_a m_u.1)\\
	&=&\big(am_u.1+L_a(m_u.1)\big)\\
	&=&\alpha(a)\big(\overline\alpha\circ\eta_P(u)\big)+(\overline\alpha\circ L_a)\big(\eta_P(u)\big),
\end{eqnarray*}
whereas:
\allowdisplaybreaks
\begin{eqnarray*}
	\eta_Q\circ\wt\alpha(au)
	&=&\eta_Q\big(\alpha(a)\wt\alpha(u)\big)\\
	&=&\alpha(a)\big((\eta_Q\circ\wt\alpha)(u)\big)+L_{\alpha(a)}\big(\eta_Q\circ\wt\alpha(u)\big)\\
	&=&\alpha(a)\big((\overline\alpha\circ\eta_P)(u)\big)+L_{\alpha(a)}\big(\overline\alpha\circ\eta_P(u)\big).
\end{eqnarray*}
It remains to show that the following identity holds:
\begin{equation}
	\overline\alpha\circ L_a=L_{\alpha(a)}\circ \overline \alpha,
\end{equation}
which is easily proven by direct computation on any argument $v=a_1\cdots a_r\in S(P)$.
\end{proof}

\begin{cor}\label{dipterous}
Let $\#$ be the product on $S(P)$ defined by $u\# v:=\eta_P\big(\eta_P^{-1}(u)\eta_P^{-1}(v)\big)$, and similarly on $S(Q)$. Then for any pre-Lie morphism $\alpha:P\to Q$, the map $\overline\alpha:S(P)\to S(Q)$ is a unital algebra morphism for both products $\cdot$ (commutative) and $\#$ (noncommutative in general).
\end{cor}

Recall \cite{ChaLiv} that the free pre-Lie algebra generated by $\Omega$ is the linear span $\Cal T^\Omega$ of $\Omega$-decorated rooted trees. The pre-Lie product $s\to t$ of two trees is given by grafting the tree $s$ successively at every vertex of $t$ and taking the sum. In this particular case, the $\#$ product is known as the Grossman--Larson product on rooted forests \cite{GL}. It is dual to the coproduct $\Delta$ in the sense that we have:
\begin{equation}
\langle F\# G, \,H\rangle=\frac{|\mop{Aut} F||\mop{Aut} G|}{|\mop{Aut} H|}\langle F\otimes G,\, \Delta H\rangle.
\end{equation}
Here $|\mop{Aut} F|$ is the symmetry factor of $F$, and similarly for $G$ and $H$. The pairing is defined by $\langle F,G\rangle=\delta_F^G$, where $\delta$ is the Kronecker delta.

\begin{prop}\label{dipterous-bis}
Let $A$ be a pre-Lie algebra, and let $\Cal S(A)$ be its symmetric algebra, endowed with the free commutative product $\cdot$ and the product $\#$ defined above. Let $\bm a=(a_\omega)_{\omega\in\Omega}$ be a collection of elements of $A$. There exists a unique linear map $\Cal F_{\bm a}:\Cal H^\Omega_<\to S(A)$ which is a unital algebra morphism for both products $\cdot$ and $\#$, such that $\Cal F_{\bm a}(\bullet_\omega)=a_\omega$.
\end{prop}

\begin{proof}
By freeness property of the pre-Lie algebra $(\Cal T^\Omega,\to)$, the restriction of $\Cal F_{\bm a}$ to $\Cal T^\Omega$ can be defined as the unique pre-Lie algebra morphism from $\Cal T^\Omega$ into $A$ such that $\Cal F_{\bm a}(\bullet_\omega)=a_\omega$. We can then extend it multiplicatively (with respect to the commutative products $\cdot$ of both symmetric algebras) from $\Cal H^\Omega_<$ into $S(A)$. Corollary \ref{dipterous} ensures that $\Cal F_{\bm a}$ also respects the $\#$ products.
\end{proof}

\subsection{$B$-series and $S$-series}
\label{ssect:BSseries}

Let $\Omega$ be a set, let $A$ be a left pre-Lie algebra, let $\bm a=(a_\omega)_{\omega\in\Omega}$ be a collection of elements of $A$ indexed by $\Omega$, and let $\Cal F_{\bm a}:\Cal H^\Omega_<\to S(A)$ be the bi-morphism defined by Proposition \ref{dipterous-bis}. The collection $\big(\Cal F_{\bm a}(F)\big)$ is an arborescent comould in the sense of \cite{EV}. Any arborescent mould $M$ gives rise, by \textsl{arborescent mould-comould contraction}, to the \textsl{$S$-series} \cite{Mu}
\begin{equation}
	\sum_{F\, \Omega\smop{-decorated} \atop \smop{rooted forest}}\frac{M^F}{|\mop{Aut}F|}\Cal F_{\bm a}(F),
\end{equation}
which makes sense if $A$ is endowed with a suitable topology such that the sum above is convergent. The corresponding $B$-series \cite{HLW} is the restriction to rooted trees:
\begin{equation}
	\sum_{T\, \Omega\smop{-decorated} \atop \smop{rooted tree}}\frac{M^T}{|\mop{Aut}T|}\Cal F_{\bm a}(T),
\end{equation}
which belong to $A$. We will mostly look at the tautological case, when $A=\Cal T^\Omega$ and $\Cal F_{\bm a}=\mop{Id}$. In particular we define:
\begin{equation}
	S^M:=\sum_{F\, \Omega\smop{-decorated} \atop \smop{rooted forest}}\frac{M^F}{|\mop{Aut}F|}F,
\end{equation}
which makes sense in the completion of $\Cal H^\Omega_<$ for the grading defined by the number of vertices. The mould is obviously determined by its $S$-series. The Grossman--Larson product extends to the completion, and we have:
\begin{equation}\label{GL}
	S^{M\times N}=S^M\#S^N.
\end{equation}

\subsection{Product and composition of arborescent moulds}
\label{ssect:abomould}

The product is defined by dualizing the coproduct $\Delta$. It can be seen as the completion of the Grossman--Larson product of Paragraph \ref{sect:pre-lie} (see proof of Proposition \ref{mould-calculus-arbo} below). The analogue of the diamond composition is obtained by dualizing the internal coproduct $\Gamma$, namely:
\begin{equation}
(M\times N)^F=(M\otimes N)^{\Delta F},\hskip 15mm (M\diamond N)^F=(M\otimes N)^{\Gamma F},
\end{equation}
which yields:
\allowdisplaybreaks
\begin{eqnarray}
	(M\times N)^F 
	&=&\sum_{V_1\sqcup V_2=\Cal V(F) \atop V_1<V_2}M^{F\srestr{V_1}}N^{F\srestr{V_2}},\\
	(M\diamond N)^F 
	&=&\sum_{G\smop{ covering} \atop \smop{subforest of }F}M^{F/G}N^G. \label{arbo-diamond}
\end{eqnarray}
As a direct consequence of Theorems \ref{coprod-c} and \ref{arbo-interne}, and from the fact that the contracting arborification $\mathfrak a$ is a Hopf algebra morphism, the following holds:

\begin{thm}\label{arbo-assoc}
The product and the diamond composition of arborescent moulds are associative, and we moreover have for any ordinary moulds $M$ and $N$:
\begin{equation}
	(M\times N)_<=M_<\times N_<,\hskip 15mm (M\diamond N)_<=M_<\diamond N_<.
\end{equation}
\end{thm}

The composition of arborescent moulds appears in \cite{EV95}, see Formula (11.48) therein. It also appears in E.~Vieillard-Baron's thesis (\cite[Paragraph 5.6]{VB}, see also \cite{VB15}). It is given for two arborescent moulds $M$ and $N$ by:
\begin{equation}\label{arbo-circ}
	(M\circ N)^F:=\sum_{\scriptstyle G\smop{ covering subforest}  \smop{ of } F,\atop \scriptstyle G=G_1\cdots G_r}M^{F/G}N^{G_1}\cdots N^{G_r},
\end{equation}
where $F$ is any $\Omega$-decorated forest, and where the $G_j$'s are the connected components of the covering subforest $G$. From \eqref{arbo-diamond} and \eqref{arbo-circ}, the two compositions $M\diamond N$ and $M\circ N$ coincide when $N$ is \textsl{separative}, i.e. when $N$ is a unital algebra morphism.

\begin{prop}\label{mould-calculus-arbo}
Both operations are associative, and composition distributes on the right over multiplication, namely:
$$
	(M\times M')\circ N=(M\circ N)\times(M'\circ N)
$$
for any triple of arborescent moulds $(M,M',N)$. The unit for the product is the mould $\varepsilon$ defined by $\varepsilon^\emptyset=1$ and $\varepsilon^{F}=0$ for any nontrivial $\Omega$-decorated rooted forest $F$. The mould $I_<$, defined by $I_<^{\bullet_\omega}=1$ for any $\omega \in \Omega$ and $I_<^{F}=0$ for $F=\emptyset$ or $F$ forest with at least two vertices, is a unit on the right.
\end{prop}

\begin{proof}
The unit properties for $\varepsilon$ and $I_<$ are immediate. Note that $I_<$ cannot be a unit on the left. Indeed, for $\omega_1,\omega_2\in\Omega$ we have:
$$
	(I_<\circ N)^{\bullet_{\omega_1}\bullet_{\omega_2}}
	=I_<^{\bullet_{\omega_1}\bullet_{\omega_2}}N^{\bullet_{\omega_1}}N^{\bullet_{\omega_2}}=0,
$$
which differs in general from $N^{\bullet_{\omega_1}\bullet_{\omega_2}}$. The associativity of the arborescent mould product directly comes from \eqref{GL}. Associativity of arborescent composition can be checked directly: here $F\subseteq G$ means that $G$ is a covering subforest of $F$, and $G=G_1\cdots G_r$ means that the $G_j$'s are the connected components of $G$. For three arborescent moulds $M,N,P$ and for any decorated forest $F$ we have:
\allowdisplaybreaks
\begin{eqnarray*}
	\big((M\circ N)\circ P\big)^F
	&=&\sum_{r\ge 1}\sum_{G\subseteq F \atop G=G_1\cdots G_r}(M\circ N)^{F/G}P^{G_1}\cdots P^{G_r}\\
	&=&\sum_{r\ge s\ge 1}\sum_{{\scriptstyle G\subseteq F,\,G=G_1\cdots G_r\atop \scriptstyle \wt H\subseteq F/G,\, \wt H=\wt H_1\cdots \wt H_s}}
M^{F/G{\textstyle /} \wt H}N^{\wt H_1}\cdots N^{\wt H_s}P^{G_1}\cdots P^{G_r}\\
	&=&\sum_{r\ge s\ge 1}\sum_{H\subseteq G\subseteq F \atop {H=H_1\cdots H_s,\,G=G_1\cdots G_r}}
M^{F/H}N^{H_1/G\cap H_1}\cdots N^{H_s/G\cap H_s}P^{G_1}\cdots P^{G_r}\\
	&=&\sum_{s\ge 1}\sum_{H\subseteq F \atop H=H_1\cdots H_s}M^{F/H}(N\circ P)^{H_1}\cdots (N\circ P)^{H_s}\\
	&=&\big(M\circ(N\circ P)\big)^F.
\end{eqnarray*}
The distributivity property is also checked by a direct computation:
\allowdisplaybreaks
\begin{eqnarray*}
	\big((M\circ N)\times(M'\circ N)\big)^F
	&=&\sum_{W_1\sqcup W_2=\Cal V(F) \atop W_1<W_2}(M\circ N)^{F\srestr{W_1}}(M'\circ N)^{F\srestr{W_2}}\\
	&&\hskip -35mm =\sum_{W_1\sqcup W_2=\Cal V(F) \atop W_1<W_2}\ \sum_{s+t\ge 1}
	\sum_{{\scriptstyle G\subseteq F\srestr{W_1}, H\subseteq F\srestr{W_2}\atop\scriptstyle G=G_1\cdots G_s,\, H=H_1\cdots H_t}}
	M^{F\srestr{W_1}{\textstyle /} G}{M'}^{F\srestr{W_2}{\textstyle /} H}N^{G_1}\cdots N^{G_s}N^{H_1}\cdots N^{H_t}\\
	&&\hskip -35mm =\sum_{W_1\sqcup W_2=\Cal V(F) \atop W_1<W_2}\ 
	\sum_{r\ge 1}\sum_{{\scriptstyle G\subseteq F,\, G=G_1\cdots G_r\atop\scriptstyle G_j\subseteq F\srestr{W_1} \smop{ or } G_j\subseteq F\srestr{W_2}}}
	M^{F\srestr{W_1}{\textstyle /} G\cap F\srestr{W_1}}{M'}^{F\srestr{W_2}{\textstyle /} G\cap F\srestr{W_2}}N^{G_1}\cdots N^{G_r}\\
	&=&\sum_{r\ge 1}\ \sum_{G\subseteq F \atop G=G_1\cdots G_r}
	\sum_{V_1\sqcup V_2=\Cal V(F/G) \atop V_1<V_2}M^{(F/G)\srestr{V_1}}{M'}^{(F/G)\srestr{V_2}}N^{G_1}\cdots N^{G_r}\\
	&=&\sum_{r\ge 1}\ \sum_{G\subseteq F \atop G=G_1\cdots G_r} (M\times M')^{F/G}N^{G_1}\cdots N^{G_r}\\
	&=&\big((M\times M')\circ N\big)^F.
\end{eqnarray*}
\end{proof}

\begin{rmk}\rm
There seems to be no available proof of Proposition \ref{mould-calculus-arbo} parallel to the one of Proposition \ref{mould-calculus}. In other words, associativity of the arborescent composition $\circ$ cannot be directly derived from the associativity of $B$-series substitution \cite{CHV}, although the two results are closely related.
\end{rmk}
\begin{rmk}\rm
Composition of ordinary moulds does not correspond to composition of arborescent moulds via contracting arborification: for example, an easy computation gives:
$$(M\circ N)_<^{\bullet_{\omega_1}\bullet_{\omega_2}}-(M_<\circ N_<)^{\bullet_{\omega_1}\bullet_{\omega_2}}
=M^{[\omega_1+\omega_2]}(N^{\omega_1\sqshu\omega_2}-N^{\omega_1}N^{\omega_2}).$$
By Theorem \ref{arbo-assoc}, the identity $(M\circ N)_<=M_<\circ N_<$ however holds when $N$ is symmetrel. An explicit formula for $(M\circ N)_<$ is given in \cite{Me06}, see Equation (5.8) therein.
\end{rmk}

Certainly many more interesting results are at reach in various contexts by
properly using in combination a first coproduct and a second internal one transposed from the operation of mould composition (and the present
text illustrates that such transpositions are possible, with some care). A quite striking achievement in these respects, is the one recently obtained
by Lo\"\i c Foissy: the chromatic polynomial of a graph is characterized as the
only polynomial that is compatible by two biagebras in interaction, which, in this situation, correspond to the two coproducts considered here (see \cite{F16} and the bibliography therein).


\end{document}